\documentclass[onefignum,onetabnum,final]{siamart220329}
\usepackage{amsmath,amssymb}
\newcommand{\qedhere}{{}}

\usepackage{algorithm}
\usepackage{algorithmic}

\usepackage[shortlabels]{enumitem}
\usepackage{comment}
\usepackage[protrusion=true,tracking=false,kerning=true]{microtype}
\usepackage{booktabs}
\usepackage{graphicx}
\usepackage{bm}

\newcommand{\vb}{\mathbf{b}}

\newcommand{\vg}{\mathbf{g}}
\newcommand{\vh}{\mathbf{h}}
\newcommand{\vx}{\mathbf{x}}
\newcommand{\vy}{\mathbf{y}}
\newcommand{\vs}{\mathbf{s}}
\newcommand{\vt}{\mathbf{t}}
\newcommand{\vu}{\mathbf{u}}
\newcommand{\vv}{\mathbf{v}}
\newcommand{\vw}{\mathbf{w}}
\newcommand{\vz}{\mathbf{z}}
\newcommand{\K}{\mathcal{K}}
\newcommand{\Ksc}{({\K}^*)^\circ}
\newcommand{\R}{\mathbb{R}}
\newcommand{\Rn}{\mathbb{R}^n}
\newcommand{\T}{\top}
\newcommand{\st}{\text{s.t.}}

\newtheorem{example}[theorem]{Example}

\newcommand{\TheTitle}{Interior-point algorithms with full Newton steps for nonsymmetric convex conic optimization}
\newcommand{\TheAuthors}{ D{\'a}vid PAPP and Anita VARGA}

\headers{Full-step IPAs for nonsymmetric conic optimization}{\TheAuthors}

\title{{\TheTitle}\thanks{\textbf{Funding:} This material is based upon work supported by the National Science Foundation under Grant No. DMS-1847865. This material is based upon work supported by the Air Force Office of Scientific Research under award number FA9550-23-1-0370.}}

\author{
    D{\'a}vid PAPP\thanks{North Carolina State University, Department of Mathematics. (ORCID: 0000-0003-4498-6417) Email: \email{dpapp@ncsu.edu}.}
\and
    Anita VARGA\thanks{North Carolina State University, Department of Mathematics. (ORCID: 0000-0001-7138-2921) Email: \email{avarga@ncsu.edu}.}
}

\ifpdf
\hypersetup{
	pdftitle={\TheTitle},
	pdfauthor={\TheAuthors}
}
\fi

\begin{document}
\maketitle
\begin{abstract}
We design and analyze primal-dual, feasible interior-point algorithms (IPAs) employing full Newton steps to solve convex optimization problems in standard conic form. Unlike most nonsymmetric cone programming methods, the algorithms presented in this paper require only a logarithmically homogeneous self-concordant barrier (LHSCB) of the primal cone, but compute feasible and $\varepsilon$-optimal solutions to both the primal and dual problems in $\mathcal{O}(\sqrt\nu\log(1/\varepsilon))$ iterations, where $\nu$ is the barrier parameter of the LHSCB; this matches the best known theoretical iteration complexity of IPAs for both symmetric and nonsymmetric cone programming. The definition of the neighborhood of the central path and feasible starts ensure that the computed solutions are compatible with the dual certificates framework of (Davis and Papp, 2022). Several initialization strategies are discussed, including two-phase methods that can be applied if a strictly feasible primal solution is available, and one based on a homogeneous self-dual embedding that allows the rigorous detection of large feasible or optimal solutions. In a detailed study of a classic, notoriously difficult, polynomial optimization problem, we demonstrate that the methods are efficient and numerically reliable. Although the standard approach using semidefinite programming fails for this problem with the solvers we tried, the new IPAs compute highly accurate near-optimal solutions that can be certified to be near-optimal in exact arithmetic.
\end{abstract}



\section{Introduction}
\label{sec:intro}

We propose and analyze new interior-point algorithms (IPAs) for convex conic optimization problems, that is, problems in the form
\begin{equation} \label{eq:PD_problem}
\left.
\begin{aligned}
    \inf\;\, & \mathbf{c}^\top \mathbf{x} \\
    \st\  & A\mathbf{x} = \mathbf{b} \\
    & \mathbf{x} \in \mathcal{K}
\end{aligned}
\quad \right\} \quad \text{(P)}
\qquad \qquad \qquad \left.
\begin{aligned}
    \sup \;\, & \mathbf{b}^\top \mathbf{y} \\
    \st\;\; & A^\top \mathbf{y} + \mathbf{s} = \mathbf{c} \\
    & \mathbf{s} \in \mathcal{K}^*, \  \mathbf{y} \in \mathbb{R}^m \\
\end{aligned}
\quad \right\}  \quad\text{(D)}
\end{equation}
where $\mathbf{c} \in \mathbb{R}^n$ and $\mathbf{b} \in \mathbb{R}^m$ are given vectors, $A \in \mathbb{R}^{m \times n}$ is a given matrix with full row rank, and $\mathcal{K} \subseteq \mathbb{R}^n$ is a proper convex cone whose dual cone is denoted by $\mathcal{K}^*$. The only assumption we make about $\K$ is that its interior, denoted by $\K^\circ$, is the domain of a known logarithmically homogeneous self-concordant barrier function (LHSCB), see Appendix \ref{sec:lhscb}, whose gradient and Hessian are computable efficiently. The cone need not have any other desirable properties; in particular, we do not assume that $\K$ is symmetric, homogeneous, or hyperbolic. Nor do we make any assumptions about its dual.

The vast majority of the abundant literature on IPAs is focused on the special case when $\K=\K^*$ is a symmetric (self-dual and homogeneous) cone, also known in the earlier literature as a self-scaled cone \cite{Guler1996}; many algorithms for these problems can be seen as broad generalizations of algorithms originally developed for linear programming. \newpage For example, in \cite{NesterovTodd1997, Nesterov1999}, Nesterov and Todd proposed efficient IPAs for self-scaled cones, based on the existence of Nesterov-Todd scaling points; Schmieta and Alizadeh used the machinery of Jordan algebras to extend IPAs for linear optimization problems to second-order cone and semidefinite optimization problems \cite{SchmietaAlizadeh2003}.
Symmetric conic optimization has also enjoyed excellent software support for decades \cite{sedumi,mosek10,sdpt3,csdp, dsdp,sdpa}, with novel applications (e.g., semidefinite programs with low-rank solutions) inspiring innovative approaches even today \cite{loraine}.

In recent years, there has been a growing interest in IPAs for more general, \emph{nonsymmetric}, conic optimization \cite{Nesterov2012}. This is partly due to models in contemporary applications that cannot be formulated as symmetric conic programs, such as the applications of the exponential cone and the relative entropy cone \cite{ChandrasekaranShah2016} in quantum information theory \cite{HeSaundersonFawzi2024}, entropy maximization problems, and risk-averse portfolio optimization \cite{AhmadiJavid2012}. It has also been established that some problems that are known to be representable as second-order cone or semidefinite optimization problems can be more efficiently handled when formulated and solved as nonsymmetric conic optimization problems; $\ell_p$ cones \cite{SkajaaYe2015} and sums-of-squares cones \cite{PappYildiz2019} are but two examples; \cite{CoeyKapelevichVielma2022} is a compendium of additional ones.

The primary difficulty in further extending symmetric cone programming approaches, especially primal-dual IPAs, to nonsymmetric conic programs is that the previously identical primal and dual cones are now replaced by two different cones, one of which might be difficult to handle. For example, $\K^\circ$ might be the domain of an easily computable LHSCB, but the same might not hold for $\Ksc$, or the known barrier for the dual cone may not be the (Fenchel) conjugate of the primal barrier, which precludes the use of certain methods. For instance, the IPAs proposed by Nesterov et al.~in \cite{Nesterov1999} require both the primal and dual LHSCBs, and the size of the linear systems to be solved is doubled compared to the symmetric case. Similarly, the method in \cite{Nesterov2012} is a corrector-predictor type IPA for nonsymmetric conic optimization that uses information on the dual barrier function in its predictor steps.

Skajaa and Ye \cite{SkajaaYe2015} proposed a primal-dual predictor-corrector IPA that solves a homogeneous self-dual embedding model for the simultaneous solution of (P) and (D), and it is one of the few methods that make no assumptions on the dual cone. Their original manuscript contained inaccuracies in certain proofs, which were later addressed by Papp and Y{\i}ld{\i}z \cite{PappYildiz2017}. This algorithm is the basis of the IPAs implemented in the Matlab package alfonso \cite{PappYildiz2022} and the Julia solver Hypatia \cite{CoeyKapelevichVielma2023}.

A specialized IPA for sums-of-squares cones, which can be generalized word-for-word to spectrahedral cones, was proposed in \cite{DavisPapp2022}; the authors also show that for some spectrahedral cones this method is strictly more efficient than solving the conventional semidefinite programming formulations, although the iteration complexity of the method is somewhat higher than typical, $\mathcal{O}(\nu)$ instead of $\mathcal{O}(\sqrt\nu)$, where $\nu$ is the barrier parameter of the LHSCB used in the algorithm.

Yet another recent approach to handling difficult dual cones is the ``domain-driven'' approach by Karimi and Tun\c{c}el \cite{KarimiTuncel2024}.

Serrano introduced short-step potential reduction methods in his thesis \cite{Serrano2015}, and these algorithms were implemented in the ECOS solver \cite{DomahidiChuBoyd2013}; however, this is specialized to the second-order cone and the exponential cone. Mosek \cite{mosek10} is a highly capable commercial solver that can solve problems over the exponential cone and power cones (in addition to symmetric cones).

Tun{\c{c}}el et al.~\cite{Tuncel2001,MyklebustTuncel2014} extended the concept of Nesterov-Todd directions from symmetric cones to the non-symmetric case. These ideas were combined with the analysis of Skajaa and Ye by Dahl and Andersen to propose an IPA for exponential cone optimization in \cite{DahlAndersen2022}. The authors did not provide the theoretical convergence and complexity results, however, they reported that the IPA performed well in practice compared to ECOS. Based on the main ideas of the IPA in \cite{DahlAndersen2022}, Badenbroek and Dahl \cite{BadenbroekDahl2022} proposed an algorithm for nonsymmetric conic optimization. They used a homogeneous self-dual embedding model and proved that the IPA has polynomial complexity, although, once again, the iteration complexity is $\mathcal{O}(\nu)$ rather than $\mathcal{O}(\sqrt\nu)$.

\textbf{Our contributions.}
We introduce and analyze feasible, full-step, primal-dual IPAs that generate iterates in a narrow neighborhood of the central path. Unlike most of the methods cited above, they are applicable to general cones with a known LHSCB, and they generate $\varepsilon$-optimal primal-dual pairs without any assumptions on the dual cone. They have $\mathcal{O}(\sqrt\nu\log(1/\varepsilon))$ iteration complexity, where $\nu$ is the barrier parameter of the LHSCB of $\K$; this matches the iteration complexity of commonly used IPAs for symmetric cones. (This is a strict improvement, for instance, over the algorithms in \cite{DavisPapp2022} and \cite{BadenbroekDahl2022}.)

To address cases where no feasible starting solution is known, we show that the proposed methods can be applied in conjunction with the homogeneous self-dual embedding technique found in, e.g., \cite{deKlerkRoosTerlaky1997, deKlerk1998, Luo2000}.

We also outline two different two-phase approaches that avoid homogenization for applications in which a primal feasible point is known, but we do not have a well-centered primal-dual pair to initialize our methods with. This is also desirable, since in most IPA literature relying on self-dual embeddings, it is an unadressed (and seemingly fundamentally complicated) issue that the homogenizing variables may tend to very small positive values (not \emph{a priori} bounded away from zero), meaning that there is a theoretical gap between computing an ``$\varepsilon$-feasible and $\varepsilon$-optimal'' solution to the embedding problem and computing an ``$\varepsilon$-feasible and $\varepsilon$-optimal'' solution to the original problem \eqref{eq:PD_problem}. (The former may be far from being the latter.)

\textbf{Short-step methods and dual feasibility certificates.} Our methods are based on neighborhoods constructed from Dikin ellipsoids (often considered ``short-step'' methods); see Eq.~\eqref{eq:nbd}. Although for symmetric cone programming, primal-dual predictor-corrector methods operating in wider neighborhoods appear to be the most practical, and those are the ones used in most general-purpose solvers,
direct generalizations of these approaches seem to be exceedingly difficult due to the aforementioned reasons.

Feasible methods operating in our small neighborhood have additional advantages beyond the good theoretical complexity and the relative simplicity of analysis. One of our motivating applications is the aforementioned optimization over low-dimensional linear images of ``nice'' (e.g., symmetric) high-dimensional cones such as sums-of-squares (SOS) cones, which are low-dimensional images of the positive semidefinite cone. (In other words, these are the dual cones of spectrahedral cones.) The concept of \emph{dual certificates} introduced in \cite{DavisPapp2022} allows the rigorous certification of membership in such cones in lower time and substantially lower memory than what is required to construct an explicit SOS certificate. As we shall detail in Section \ref{sec:SOS-application}, the methods proposed in this paper automatically compute dual feasibility certificates in the sense of \cite{DavisPapp2022}, and they do so at a lower time complexity than the algorithm proposed in that paper. The same conclusion would not hold if we were working with a different neighborhood definition.

Our simple feasible IPA (Algorithm \ref{alg:simple_short_step}) is presented and analyzed in detail in Section \ref{sec:full_step}, along with its adaptive update variants, Algorithm \ref{alg:largest_update} and \mbox{Algorithm \ref{alg:real_largest_update}}. Some of the background material on barrier functions necessary for the analysis is summarized in Appendix~\ref{sec:lhscb}. We then discuss various initialization strategies in \mbox{Section \ref{sec:init}}. Finally, in \mbox{Section \ref{sec:SOS-application}}, we present a detailed numerical example demonstrating the numerical robustness and efficiency of the method.

\section{The full-step IPAs and their analysis}
\label{sec:full_step}

All variants of our algorithm are path-following IPAs that generate iterates in the same neighborhood. We start by introducing the necessary notation and then define the neighborhood of the central path.

Let $\mathcal{F}$ denote the set of primal-dual feasible solutions:
\[
\mathcal{F} = \left\{ (\mathbf{x}, \mathbf{y}, \mathbf{s})\in\mathcal{K}\times\mathbb{R}^m\times\mathcal{K}^* \mid A\mathbf{x} = \mathbf{b}, \ A^{\top}\mathbf{y} + \mathbf{s} = \mathbf{c} \right\},
\]
and let $\mathcal{F}^{\circ}$ be the set of strictly feasible solutions:
\[
\mathcal{F}^{\circ} = \left\{ (\mathbf{x}, \mathbf{y}, \mathbf{s})\in\mathcal{K}^\circ\times\mathbb{R}^m\times(\mathcal{K}^*)^\circ \mid A\mathbf{x} = \mathbf{b}, \ A^{\top}\mathbf{y} + \mathbf{s} = \mathbf{c} \right\}.
\]
Suppose that the interior point (aka.~Slater) condition is satisfied: $\mathcal{F}^{\circ} \neq \emptyset$. Then strong duality holds, and the optimality conditions can be formulated as
\begin{equation}
\left.
\begin{aligned}
(\vx,\vy,\vs)&\in\mathcal{F}\\
\vx^\T\vs&=0
\end{aligned} \right\}. \tag{OC}\label{eq:OC}
\end{equation}

We assume that an LHSCB $f: \mathcal{K}^{\circ} \rightarrow \mathbb{R}$ with barrier parameter $\nu$ is given for the primal cone $\mathcal{K}$. Let $g(\mathbf{x})$ and $H(\mathbf{x})$ denote the gradient and Hessian of $f$ at $\mathbf{x} \in \mathcal{K}^{\circ}$, respectively. With this notation in place, we define the central path problem, parametrized by the positive scalar $\tau$, called the \emph{central path parameter}, as follows:
\begin{equation}
\left.
\begin{aligned}
(\vx,\vy,\vs)&\in\mathcal{F}^\circ\\
    \mathbf{s} + \tau g(\mathbf{x}) \ & = \mathbf{0} \\
\end{aligned} \right\}. \tag{CP}\label{eq:CP}
\end{equation}
For every $\tau>0$, the solution of \eqref{eq:CP} is unique; these solutions form the \emph{primal-dual central path}. As $\tau$ tends to $0$, the central path approaches a solution of \eqref{eq:OC}.

In the analysis,  we consider the following \emph{neighborhood} of the central path:
\begin{equation}
\label{eq:nbd}
\mathcal{N}(\eta,\tau) = \left\{ (\mathbf{x}, \mathbf{y}, \mathbf{s}) \in \mathcal{F}^\circ : \| \mathbf{s} + \tau g(\mathbf{x}) \|_{\mathbf{x}}^* \leq \eta \tau \right\},
\end{equation}
where $\eta$ is the radius of the neighborhood, and $\tau$ is the central path parameter; the choice of $\eta$ will be discussed later. This definition can be considered the generalization of the classic 2-norm neighborhood for linear optimization. Note the small but important departure from how a similar neighborhood is most commonly defined in the literature: it is customary to use the \emph{normalized duality gap} $\mu := \frac{\mathbf{x}^{\top} \mathbf{s}}{\nu}$ in place of the parameter $\tau$ (for example \cite{Nesterov1994,Serrano2015,SkajaaYe2015}); thus, our variant generates iterates in a slightly wider neighborhood of the central path. The modified definition applied here was also examined in \cite{NesterovTodd1998} in the context of self-scaled cones.

Given an $(\mathbf{x},\mathbf{y},\mathbf{s})\in\mathcal{F}^\circ$, the \emph{Newton system} associated with the central path \eqref{eq:CP} is
\begin{equation}
\left.
\begin{aligned}
    A \Delta \mathbf{x} & = \mathbf{0} \\
    A^{\top} \Delta \mathbf{y} + \Delta \mathbf{s} & = \mathbf{0} \\
    \tau H(\mathbf{x}) \Delta \mathbf{x} + \Delta \mathbf{s} & = -(\mathbf{s} + \tau g(\mathbf{x}))
\end{aligned} \right\} .
\tag{NS}\label{eq:NS}
\end{equation}
Its solution ($\Delta \mathbf{x}$, $\Delta \mathbf{y}$, $\Delta \mathbf{s}$)
is called the \emph{Newton direction} at $(\vx,\vy,\vs)$.

We start with our basic method, which takes full Newton steps and reduces the central path parameter $\tau$ by a fixed fraction every iteration. The pseudo-code of the algorithm is shown in Algorithm \ref{alg:simple_short_step}. It requires an initial primal-dual iterate near the central path; Theorem \ref{theorem:compl1} prescribes parameter values for which the convergence of the algorithm (with a suitable initial iterate) is guaranteed, and in Section \ref{sec:init} we discuss different strategies (two-phase variants, embeddings, etc.) for when a well-centered initial iterate is not readily available.

\begin{algorithm}[h]
\caption{A full-step IPA for nonsymmetric cone programming}
\label{alg:simple_short_step}
\begin{algorithmic}
\vspace{0.5em}
\STATE \textbf{inputs:}
$A \in \mathbb{R}^{m \times n}$ with full row rank; $\mathbf{b} \in \mathbb{R}^m$;  $\mathbf{c} \in \mathbb{R}^n$;
\\ \hspace{1em}an oracle for computing the Hessian $H$ of the $\nu$-LHSCB $f: \ \mathcal{K}^{\circ} \rightarrow \mathbb{R}$;
\\ \hspace{1em}an initial $\tau_0>0$ and
$(\mathbf{x}_0, \mathbf{y}_0, \mathbf{s}_0) \in \mathcal{F}^\circ$ such that $\| \mathbf{s}_0 + \tau_0 g(\mathbf{x}_0) \|_{\mathbf{x}_0}^* \leq \eta \tau_0$;
\\ \hspace{1em}$\varepsilon > 0$ optimality tolerance.

\vspace{0.5em}
\STATE \textbf{parameters:}
neighborhood parameter $\eta \in (0, 1)$;
update parameter $\vartheta$.
\vspace{0.5em}

\STATE \textbf{output:} a triplet $(\vx,\vy,\vs)\in\mathcal{F}^\circ$ satisfying $\vx^\T\vs\leq\varepsilon$.

\vspace{1em}
\STATE Initialize $\mathbf{x} := \mathbf{x}_0$, $\mathbf{y} := \mathbf{y}_0$, $\mathbf{s} := \mathbf{s}_0$, $\tau := \tau_0$.
\vspace{0.5em}
\WHILE{$\mathbf{x}^\top \mathbf{s} > \varepsilon$}
    \STATE Solve \eqref{eq:NS} for $(\Delta \mathbf{x}, \Delta \mathbf{y}, \Delta \mathbf{s})$.
    \STATE Set $\mathbf{x} := \mathbf{x} + \Delta \mathbf{x}$, $\mathbf{y} := \mathbf{y} + \Delta \mathbf{y}$, $\mathbf{s} := \mathbf{s} + \Delta \mathbf{s}$.
    \STATE Set $\tau := (1 - \vartheta) \tau$.
\ENDWHILE
\vspace{0.5em}
\STATE \textbf{return} $(\mathbf{x}, \mathbf{y}, \mathbf{s})$.
\vspace{0.5em}
\end{algorithmic}
\end{algorithm}

The Newton system \eqref{eq:NS} is identical to the system used for determining the corrector directions in \cite{SkajaaYe2015} and \cite{PappYildiz2017}, and parts of their analyses are also applicable here. Similar bounds were also used in \cite{Serrano2015}. To keep the paper self-contained, we restate these results briefly as part of Lemma \ref{lemma:sd_properties} below. The main difference in our analysis (throughout the paper) is that we do not use the (normalized) duality gap $\mu$ as the central path parameter, as is customary in most short-step IPAs. Instead, a separate parameter $\tau$ is used, which in Algorithm \ref{alg:simple_short_step} is prescribed to converge geometrically to zero. Therefore, we have to establish separately that the duality gap also converges to $0$ and at what rate; we do so in Lemma \ref{lemma:DG}.

The following lemma summarizes some elementary results from the literature regarding the Newton directions.

\begin{lemma}[\cite{PappYildiz2017, Serrano2015,SkajaaYe2015}]\label{lemma:sd_properties}
Let $(\mathbf{x}, \mathbf{y}, \mathbf{s}) \in \mathcal{N}(\eta, \tau)$ and let $(\Delta \mathbf{x}, \Delta \mathbf{y}, \Delta \mathbf{s})$ be the solution of \eqref{eq:NS}. Then
\begin{enumerate}[label=(\roman*)]
    \item  $\Delta \mathbf{x}^{\top} \Delta \mathbf{s} = 0$, \label{sd_prop:first}
    \item $
\| \Delta \mathbf{x} \|_{\mathbf{x}} \leq \eta$ and $\| \Delta \mathbf{s} \|_{\mathbf{x}}^* \leq \eta \tau
$,  \label{sd_prop:second}
    \item $\| \mathbf{s} + \tau g(\mathbf{x}) + \Delta \mathbf{s} \|_{\mathbf{x}}^* \leq \eta \tau$. \label{sd_prop:third}
\end{enumerate}
\end{lemma}

\begin{proof}
    \begin{enumerate}[label=(\roman*)]
        \item Using the first two equations of \eqref{eq:NS}, we get
    \[ 0 = \Delta \mathbf{x}^{\top} (A^{\top} \Delta \mathbf{y} + \Delta \mathbf{s}) = \Delta \mathbf{y}^{\top} A \Delta \mathbf{x} + \Delta \mathbf{x}^{\top} \Delta \mathbf{s} = \Delta \mathbf{x}^{\top} \Delta \mathbf{s}. \]
        \item Using the neighborhood definition, the orthogonality of $\Delta \mathbf{x}$ and $\Delta \mathbf{s}$, and the third equality of the Newton system, we can derive the upper bounds the following way:
\[ \begin{aligned}
\eta^2 \tau^2 &\geq \left( \| \mathbf{s} + \tau g(\mathbf{x}) \|_{\mathbf{x}}^* \right)^2 = \left( \| \tau H(\mathbf{x}) \Delta \mathbf{x} + \Delta \mathbf{s} \|_{\mathbf{x}}^* \right)^2 \\
&= \tau^2 \Delta \mathbf{x}^\top H(\mathbf{x}) \Delta \mathbf{x} + \Delta \mathbf{s}^\top H(\mathbf{x})^{-1} \Delta \mathbf{s}
= \tau^2 \| \Delta \mathbf{x} \|_{\mathbf{x}}^2 + (\| \Delta \mathbf{s} \|_{\mathbf{x}}^*)^2.
\end{aligned}
\]

As both terms on the right-hand side are nonnegative, it follows that $\| \Delta \mathbf{x} \|_{\mathbf{x}} \leq \eta$ and $\| \Delta \mathbf{s} \|_{\mathbf{x}}^* \leq \eta \tau$.
    \item Using the third equation of \eqref{eq:NS} and the upper bound on $\|\Delta \mathbf{x}\|_{\mathbf{x}}$ from \ref{sd_prop:second}, we get
    \[
\| \mathbf{s} + \tau g(\mathbf{x}) + \Delta \mathbf{s} \|_{\mathbf{x}}^* = \| \tau H(\mathbf{x}) \Delta \mathbf{x} \|_{\mathbf{x}}^*
= \tau \| \Delta \mathbf{x} \|_{\mathbf{x}} \leq \eta \tau.
  \]

    \end{enumerate}
\end{proof}

In our next Lemma, we show that with a suitable choice of parameter values, the iterate after a full Newton step remains in the required neighborhood $\mathcal{N}(\eta,\tau^+)$ for our chosen $\tau^+ = (1-\vartheta)\tau$.

\begin{lemma}
\label{lemma:zplus-in-Nthetaplus}
Let $(\mathbf{x}, \mathbf{y}, \mathbf{s}) \in \mathcal{N}(\eta, \tau)$, let $(\Delta \mathbf{x}, \Delta \mathbf{y}, \Delta \mathbf{s})$ be the solution of \eqref{eq:NS}, and let $(\mathbf{x}^+, \mathbf{y}^+, \mathbf{s}^+) = (\mathbf{x}, \mathbf{y}, \mathbf{s}) + (\Delta \mathbf{x}, \Delta \mathbf{y}, \Delta \mathbf{s})$. If $0 < \eta \leq \frac{1}{4}$ and $
\vartheta = \frac{\eta / 2}{\sqrt{\nu} + 1}$,
then the new point satisfies $(\mathbf{x}^+, \mathbf{y}^+, \mathbf{s}^+) \in \mathcal{N}(\eta, \tau^+)$ with $\tau^+ = (1 - \vartheta) \tau$.
\end{lemma}

\begin{proof}
    Let us define the function $\varphi : \mathcal{K}^\circ \to \mathbb{R}$:
\[
\varphi(\mathbf{v}) = \frac{1}{\tau} (\mathbf{s} + \Delta \mathbf{s})^\top \mathbf{v} + f(\mathbf{v}),
\]
where $f$ is the known LHSCB for $\mathcal{K}^\circ$, with gradient $g$ and Hessian $H$.

Then $\varphi$ is self-concordant, and its Hessian is also $H$.
Let $\mathbf{n}_{\varphi}(\mathbf{v})$ denote the (unconstrained) Newton step for $\varphi$ at $\mathbf{v} \in \mathcal{K}^\circ$:
\[
\mathbf{n}_{\varphi}(\mathbf{v}) = -\tau^{-1} H(\mathbf{v})^{-1} \left( \mathbf{s} + \Delta \mathbf{s} + \tau g(\mathbf{v}) \right).
\]
Thus,
\[
\mathbf{n}_{\varphi}(\mathbf{x}) = -\tau^{-1} H(\mathbf{x})^{-1} \left( \mathbf{s} + \Delta \mathbf{s} + \tau g(\mathbf{x}) \right) = \Delta \mathbf{x},
\]
and therefore,
\[
\mathbf{x}^+ = \mathbf{x} + \Delta \mathbf{x} = \mathbf{x} + \mathbf{n}_{\varphi}(\mathbf{x}).
\]

For any $\mathbf{v} \in \mathcal{K}^\circ$,
\[
\|\mathbf{n}_{\varphi}(\mathbf{v})\|_{\mathbf{v}} = \left\| \tau^{-1} H(\mathbf{v})^{-1} \left( \mathbf{s} + \Delta \mathbf{s} + \tau g(\mathbf{v}) \right) \right\|_{\mathbf{v}}
= \tau^{-1} \|\mathbf{s} + \Delta \mathbf{s} + \tau g(\mathbf{v})\|_{\mathbf{v}}^*.
\]
For $\mathbf{v} = \mathbf{x}$, $\|\mathbf{n}_{\varphi}(\mathbf{x})\|_{\mathbf{x}} = \|\Delta \mathbf{x}\|_{\mathbf{x}} \leq \eta$, using Lemma \ref{lemma:sd_properties}, part \ref{sd_prop:second}. As $\eta < 1$, we can use \ref{prop:scb_NS}:
\begin{equation} \label{eq:UB_NS}
    \tau^{-1} \|\mathbf{s} + \Delta \mathbf{s} + \tau g(\mathbf{x}^+)\|^*_{\mathbf{x}^+} = \|\mathbf{n}_{\varphi}(\mathbf{x}^+)\|_{\mathbf{x}^+}
\leq \left( \frac{\|\mathbf{n}_{\varphi}(\mathbf{x})\|_{\mathbf{x}}}{1 - \|\mathbf{n}_{\varphi}(\mathbf{x})\|_{\mathbf{x}}} \right)^2 \leq \frac{\eta^2}{(1 - \eta)^2}.
\end{equation}

Thus,
\begin{align*}
\frac{1}{\tau^+} \left\| \mathbf{s}^+ + \tau^+ g(\mathbf{x}^+) \right\|_{\mathbf{x}^+}^*
&= \frac{1}{1 - \vartheta} \tau^{-1} \| \mathbf{s} + \Delta \mathbf{s} + \tau^+ g(\mathbf{x}^+) \|_{\mathbf{x}^+}^* \\
&\leq \frac{1}{1 - \vartheta} \tau^{-1} \| \mathbf{s} + \Delta \mathbf{s} + \tau g(\mathbf{x}^+)\|_{\mathbf{x}^+}^* + \frac{\vartheta}{1 - \vartheta} \| g(\mathbf{x}^+) \|^*_{\mathbf{x}^+} \\
&\stackrel{\eqref{eq:UB_NS},\ref{prop:lhscb_norm}}{\leq} \frac{1}{1 - \vartheta}  \frac{\eta^2}{(1 - \eta)^2} + \frac{\vartheta}{1 - \vartheta} \sqrt{\nu} \\
&= \frac{\sqrt{\nu} + 1}{\sqrt{\nu} + 1 - \eta/2} \cdot \frac{\eta^2}{(1 - \eta)^2} + \frac{\eta / 2}{\sqrt{\nu} + 1 - \eta/2} \sqrt{\nu} \\
&\leq \frac{2}{2 - \eta / 2} \cdot \frac{\eta^2}{(1 - \eta)^2} + \frac{\eta}{2}.
\end{align*}

The last expression is less than $\eta$ if $\eta \leq \frac{1}{4}$ holds.
\end{proof}

In our final Lemma of this section, we give bounds on the duality gap after each iteration to measure our progress toward optimality. This will help establish linear convergence and an $\mathcal{O} (\sqrt{\nu}\log(1/\varepsilon))$ iteration complexity.

\begin{lemma} \label{lemma:DG}
Let $(\mathbf{x}, \mathbf{y}, \mathbf{s}) \in \mathcal{N}(\eta, \tau)$ and let $(\Delta \mathbf{x}, \Delta \mathbf{y}, \Delta \mathbf{s})$ be the solution of \eqref{eq:NS}. We have the following lower and upper bounds on the duality gap after a full Newton step:
    \[ \tau(\nu - \eta^2) \leq (\mathbf{x}^+)^\top \mathbf{s}^+ \leq \tau \nu. \]
\end{lemma}

\begin{proof}
Using \ref{sd_prop:first} from Lemma \ref{lemma:sd_properties},
    \[ (\mathbf{x}^+)^\top \mathbf{s}^+ = (\mathbf{x} + \Delta \mathbf{x})^\top (\mathbf{s} + \Delta \mathbf{s}) = \mathbf{x}^\top \mathbf{s} + \mathbf{s}^\top \Delta \mathbf{x} + \mathbf{x}^\top \Delta \mathbf{s}. \]
From the third equation of \eqref{eq:NS}, we have
\begin{align*}
    \Delta \mathbf{x}^\top \mathbf{s} &= \Delta \mathbf{x}^\top \mathbf{s} + \Delta \mathbf{x}^\top \Delta \mathbf{s} \\
&= \Delta \mathbf{x}^\top \mathbf{s} - \Delta \mathbf{x}^\top (\mathbf{s} + \tau g(\mathbf{x}) + \tau H(\mathbf{x}) \Delta \mathbf{x}) \\
&= -\tau (\Delta \mathbf{x}^\top g(\mathbf{x}) + \|\Delta \mathbf{x}\|_{\mathbf{x}}^2).
\end{align*}
Similarly, we can express $\mathbf{x}^\top \Delta \mathbf{s}$ as
\begin{align*}
\mathbf{x}^\top \Delta \mathbf{s} &= - \mathbf{x}^\top \left( \mathbf{s} + \tau g(\mathbf{x}) + \tau H(\mathbf{x}) \Delta \mathbf{x} \right) \\
&\stackrel{\ref{prop:lhscb_Hessian},\ref{prop:lhscb_gradient}}{=} -\mathbf{x}^\top \mathbf{s} + \tau \nu + \tau \Delta \mathbf{x}^\top g(\mathbf{x}).
\end{align*}

Combining the previous equations, we can write $(\mathbf{x}^+)^\top \mathbf{s}^+$ as
\begin{align*}
(\mathbf{x}^+)^\top \mathbf{s}^+ &= \mathbf{x}^\top \mathbf{s} - \tau \Delta \mathbf{x}^\top g(\mathbf{x}) - \tau \|\Delta \mathbf{x}\|_{\mathbf{x}}^2 - \mathbf{x}^\top \mathbf{s} + \tau \nu + \tau \Delta \mathbf{x}^\top g(\mathbf{x}) \\
&= \tau (\nu -  \|\Delta \mathbf{x}\|_{\mathbf{x}}^2).
\end{align*}
Using part \ref{sd_prop:second} of Lemma \ref{lemma:sd_properties}, we get
\[ \tau(\nu-\eta^2) \leq (\mathbf{x}^+)^\top \mathbf{s}^+ \leq \tau\nu.  \]

\end{proof}

We are now ready to show that Algorithm \ref{alg:simple_short_step} produces an $\varepsilon$-optimal solution in $\mathcal{O} \left( \sqrt{\nu} \log {1/\varepsilon} \right)$ iterations.

\begin{theorem} \label{theorem:compl1}

If $0 < \eta \leq \frac{1}{4}$ and $
\vartheta = \frac{\eta / 2}{\sqrt{\nu} + 1}$, then Algorithm \ref{alg:simple_short_step} produces a feasible solution to \eqref{eq:PD_problem} with duality gap below $\varepsilon>0$ in no more than $\frac{2}{\eta}\left(\sqrt{\nu}+1 \right)  \ln \left(\frac{\tau_0 \nu}{\varepsilon}\right) +1$
iterations.

In particular, if the initial point satisfies $\tau_0 = \frac{\vx_0^\T\vs_0}{\nu}$, as in the initialization strategies discussed in Section \ref{sec:init}, then the required number of iterations is
\[\mathcal{O}\left(\sqrt\nu\ln\frac{\vx_0^\T\vs_0}{\varepsilon}\right).\]
\end{theorem}

\begin{proof}
    At the beginning of the $k^{\text{th}}$ iteration, $\tau = \tau_k = (1-\vartheta)^{k-1} \tau_0$. According to Lemma \ref{lemma:DG}, we have
    \[ \mathbf{x}_k^{\top} \mathbf{s}_k  \leq \tau_k \nu = (1-\vartheta)^{k-1} \tau_0 \nu. \]
    We would like to find the smallest $k$ for which  $(1-\vartheta)^{k-1} \tau_0 \nu \leq \varepsilon$ holds, which is equivalent to
    \[ (k-1) \ln (1-\vartheta) + \ln ( \tau_0 \nu ) \leq \ln \varepsilon.\]
    Using the inequality $-\ln(1 - \vartheta) \geq \vartheta$ for $\vartheta \in [0,1)$, we can require the fulfillment of the stronger inequality
    \[ -(k-1) \vartheta + \ln (\tau_0 \nu ) \leq \ln \varepsilon,\]
which is true if
\[
k \geq \frac{1}{\vartheta} \ln \left( \frac{\tau_0 \nu}{\varepsilon} \right) +1
= \frac{2}{\eta} \left( \sqrt{\nu} + 1 \right)  \ln \left(  \frac{\tau_0\nu}{\varepsilon} \right) +1.
\]

\end{proof}

\subsection{Adaptive update strategies}
\label{sec:improved-update-strategies}

In this section, we consider more efficient strategies for updating the parameter value $\tau$.

After taking a full Newton step, we can find the smallest $\tau^+$ for which the new point belongs to the neighborhood $\mathcal{N}(\eta,\tau^+)$, and set this $\tau^+$ as the updated value of $\tau$. (See Algorithm \ref{alg:largest_update}.) This update is only marginally more expensive than the simple update formula used in Algorithm \ref{alg:simple_short_step}: although it requires the computation of the Hessian of the barrier function at the new point, that Hessian is needed anyway for the next iteration. With this Hessian computed, the updated $\tau^+$ can be calculated by solving a scalar quadratic equation; it can be explicitly expressed as
\begin{equation} \label{eq:impr_update}
\tau^+ = \frac{1}{\nu - \eta^2} \left((\vx^+)^\T\vs^+-\sqrt{((\mathbf{x}^+)^{\top} \mathbf{s}^+)^2 - (\nu - \eta^2) \left( \| \mathbf{s}^+ \|^{*}_{\mathbf{x}_+} \right)^2}\right).
\end{equation}

It follows automatically that the complexity of Algorithm \ref{alg:largest_update} is no greater than the complexity of Algorithm \ref{alg:simple_short_step}, that is, Theorem \ref{theorem:compl1} applies verbatim to our second algorithm as well.

In \cite{DavisPapp2022}, the authors proposed an IPA (Algorithm 3.1) to determine (weighted) sums-of-squares certificates for positive polynomials with a similar update strategy. In our notation, the method of \cite{DavisPapp2022} defines the $\eta$-neighborhood of the central path as the set of strictly feasible solutions satisfying $\| \mathbf{s} + g(\mathbf{x}) \|_{\mathbf{x}}^* \leq \eta$, reminiscent but different from both our neighborhood definition and the customary one in the literature. The different neighborhood definition considered in this paper makes it possible to improve the complexity of the algorithm: the iteration complexity of the algorithm in \cite{DavisPapp2022} is $\mathcal{O}(\nu)$, while we have seen that our algorithms have an $\mathcal{O}(\sqrt{\nu})$ iteration complexity.

\begin{algorithm}[h]
\caption{Full-step IPA with a simple adaptive update strategy}
\label{alg:largest_update}
\begin{algorithmic}
\vspace{0.5em}
\STATE \textbf{inputs:}
$A \in \mathbb{R}^{m \times n}$ with full row rank; $\mathbf{b} \in \mathbb{R}^m$;  $\mathbf{c} \in \mathbb{R}^n$;
\\ \hspace{1em}an oracle for computing the Hessian $H$ of the $\nu$-LHSCB $f: \ \mathcal{K}^{\circ} \rightarrow \mathbb{R}$;
\\ \hspace{1em}an initial $\tau_0>0$ and
$(\mathbf{x}_0, \mathbf{y}_0, \mathbf{s}_0) \in \mathcal{F}^\circ$ such that $\| \mathbf{s}_0 + \tau_0 g(\mathbf{x}_0) \|_{\mathbf{x}_0}^* \leq \eta \tau_0$;
\\ \hspace{1em}$\varepsilon > 0$ optimality tolerance.

\vspace{0.5em}
\STATE \textbf{parameters:}
neighborhood parameter $\eta \in (0, 1)$;
update parameter $\vartheta$.
\vspace{0.5em}

\STATE \textbf{output:} a triplet $(\vx,\vy,\vs)\in\mathcal{F}^\circ$ satisfying $\vx^\T\vs\leq\varepsilon$.

\vspace{1em}
\STATE Initialize $\mathbf{x} := \mathbf{x}_0$, $\mathbf{y} := \mathbf{y}_0$, $\mathbf{s} := \mathbf{s}_0$, $\tau := \tau_0$
\vspace{0.5em}
\WHILE{$\mathbf{x}^\top \mathbf{s} > \varepsilon$}
    \STATE Solve \eqref{eq:NS} for $(\Delta \mathbf{x}$, $\Delta \mathbf{y}$, $\Delta \mathbf{s})$.
    \STATE Set $\mathbf{x}^+ := \mathbf{x} + \Delta \mathbf{x}$, $\mathbf{y}^+ := \mathbf{y} + \Delta \mathbf{y}$, $\mathbf{s}^+ := \mathbf{s} +  \Delta \mathbf{s}$.
    \STATE Set $\tau^+ := \min \{ \tau \mid \left\Vert \mathbf{s}^+ + \tau g(\mathbf{x}^+) \right\Vert_{\mathbf{x}^+}^* \leq \eta \tau \}$.
    \STATE Set $\tau := \tau^+$, $\mathbf{x} := \mathbf{x}^+$, $\mathbf{y} := \mathbf{y}^+$, $\mathbf{s} := \mathbf{s}^+$.
\ENDWHILE
\vspace{0.5em}
\STATE \textbf{return} $(\mathbf{x}, \mathbf{y}, \mathbf{s})$.
\end{algorithmic}
\end{algorithm}

Another known update strategy from the theory of IPAs for linear programming (LP) and linear complementarity problems (LCPs) is the \emph{largest step method} \cite{Gonzaga1997,Potra2012}. In our previous methods, $\tau$ is updated after taking the Newton step. In contrast, the largest update strategy updates $(\vx,\vy,\vs)$ and $\tau$ simultaneously by finding the lowest value $\tau^+$ for which the new point after the update remains in the neighborhood $\mathcal{N}(\eta,\tau^+)$. That is, denoting the solution of \eqref{eq:NS} by $(\Delta\vx_\tau,\Delta\vy_\tau,\Delta\vs_\tau)$ to emphasize its dependence on $\tau$,
\begin{equation} \label{eq:largest_update}
    \tau^+ = \inf \left\{ \tau>0\,\middle|\,\|\vs+\\\Delta\vs_\tau + \tau g(\vx+\Delta\vx_\tau) \|_{\vx+\Delta\vx_\tau}^* \leq \eta \tau \right\}.
 \end{equation}
The complexity of computing or approximating the solution of \eqref{eq:largest_update} depends on the LHSCB at hand; the pseudo-code in Algorithm \ref{alg:real_largest_update} is therefore formulated generically. For LPs, monotone LCPs, and weighted monotone LCPs, \eqref{eq:largest_update} can be solved in closed form for a very efficient update. For other barriers, one may solve \eqref{eq:NS} parametrically, and if $g$ and $H$ have a convenient characterization, $\tau^+$ may be approximated efficiently, without explicitly solving \eqref{eq:NS} for different values of $\tau$ in each iteration.

In the most general case, we can use line search (starting with the theoretically guaranteed $(1-\vartheta)\tau$ from Algorithm \ref{alg:simple_short_step}) to approximately solve \eqref{eq:largest_update}, although this requires the repeated solution of \eqref{eq:NS} for different values of $\tau$.

Each of these solution methods for \eqref{eq:largest_update} is theoretically sound, since all iterates remain in $\mathcal{N}(\eta,\tau)$, and the progress in each iteration is at least as large as in \mbox{Algorithm \ref{alg:simple_short_step}}. Therefore, the $\mathcal{O}(\sqrt{\nu})$ iteration complexity also holds, but if we have to resort to line search, that may come at an additional computational cost per iteration.

\begin{algorithm}[h]
\caption{Full-step IPA with the largest update strategy}
\label{alg:real_largest_update}
\begin{algorithmic}
\vspace{0.5em}
\STATE \textbf{inputs:}
$A \in \mathbb{R}^{m \times n}$ with full row rank; $\mathbf{b} \in \mathbb{R}^m$;  $\mathbf{c} \in \mathbb{R}^n$;
\\ \hspace{1em}an oracle for computing the Hessian $H$ of the $\nu$-LHSCB $f: \ \mathcal{K}^{\circ} \rightarrow \mathbb{R}$;
\\ \hspace{1em}an initial $\tau_0>0$ and
$(\mathbf{x}_0, \mathbf{y}_0, \mathbf{s}_0) \in \mathcal{F}^\circ$ such that $\| \mathbf{s}_0 + \tau_0 g(\mathbf{x}_0) \|_{\mathbf{x}_0}^* \leq \eta \tau_0$;
\\ \hspace{1em}$\varepsilon > 0$ optimality tolerance.

\vspace{0.5em}
\STATE \textbf{parameters:}
neighborhood parameter $\eta \in (0, 1)$;
update parameter $\vartheta$.
\vspace{0.5em}

\STATE \textbf{output:} a triplet $(\vx,\vy,\vs)\in\mathcal{F}^\circ$ satisfying $\vx^\T\vs\leq\varepsilon$.

\vspace{1em}
\STATE Initialize $\mathbf{x} := \mathbf{x}_0$, $\mathbf{y} := \mathbf{y}_0$, $\mathbf{s} := \mathbf{s}_0$, $\tau := \tau_0$
\vspace{0.5em}
\WHILE{$\mathbf{x}^\top \mathbf{s} > \varepsilon$}
    \STATE Set $\tau^+$ as the (approximate) solution of \eqref{eq:largest_update} and solve \eqref{eq:NS} for $(\Delta \mathbf{x}$, $\Delta \mathbf{y}$, $\Delta \mathbf{s})$ using $\tau=\tau^+$.

    \STATE Set $\tau:=\tau^+$, $\mathbf{x}^+ := \mathbf{x} + \Delta \mathbf{x}$, $\mathbf{y}^+ := \mathbf{y} + \Delta \mathbf{y}$, $\mathbf{s}^+ := \mathbf{s} +  \Delta \mathbf{s}$.
\ENDWHILE
\vspace{0.5em}
\STATE \textbf{return} $(\mathbf{x}, \mathbf{y}, \mathbf{s})$.
\end{algorithmic}
\end{algorithm}

\section{Initialization}
\label{sec:init}

Our algorithms are fundamentally \emph{feasible IPAs}, meaning that they require an initial primal-dual strictly feasible pair in the appropriate neighborhood of the central path. In this section, we address the question of finding such an initial point. First, we consider an important special case that covers one of our motivating applications outlined in the introduction, in which initialization is trivial. Then we outline two different two-phase approaches for bounded problems in which only a primal feasible solution is known, but not a primal-dual pair near the central path. Finally, we briefly demonstrate the application of a standard homogeneous self-dual embedding technique (see, e.g., \cite{deKlerkRoosTerlaky1997} for semidefinite programming) in our setting, with no assumptions about the feasibility and boundedness of the problem.

\subsection{Dual membership verification problems}
\label{sec:dual-membership-init}

Recall one of our motivating questions from the introduction: the certification of membership in the dual cone $\K^*$, without any assumption on $\K^*$ other than it being the dual cone of a cone with an LHSCB. Given a vector $\mathbf{t}\in\Rn$, we can decide whether $\mathbf{t}\in\Ksc$ by considering the following primal-dual pair of optimization problems:
\begin{equation} \label{eq:1constraint} \left.
\begin{aligned}
    & \min & & \mathbf{t}^\top \mathbf{x} \\
    &  \ \st& & \mathbf{w}^\top \mathbf{x} = 1 \\
    & & & \mathbf{x} \in \mathcal{K}
\end{aligned} \right\}
\qquad \qquad \left.
\begin{aligned}
    & \max & & y \\
    &  \ \st & & y \mathbf{w} + \mathbf{s} = \mathbf{t} \\
    & & & \mathbf{s} \in \mathcal{K}^*, \ y \in \mathbb{R}
\end{aligned} \right\},
\end{equation}
where $\mathbf{w} \in \Ksc$ is any vector of our choice form the interior of the dual cone. (Having access to such a vector $\mathbf{w}$ is not a strong assumption; recall that given any $\mathbf{x}_0 \in \K^\circ$, $-g(\mathbf{x}_0) \in \Ksc$ automatically.)

It is clear from the formulation that the dual problem is strictly feasible, since every sufficiently negative $y$ satisfies $\mathbf{t}-y\mathbf{w}\in\Ksc$. The primal problem is also strictly feasible if $\K$ is full-dimensional, since $\mathbf{w}\in\Ksc$, and therefore every $\mathbf{x}_0\in\K$ can be rescaled to satisfy the only equality constraint. Hence, strong duality holds for the problems in \eqref{eq:1constraint}, with attainment in both the primal and dual. Additionally, the optimal value is nonnegative if and only if $\mathbf{t}\in\K^*$.

Our next Lemma shows that if we have a strictly feasible primal solution $\mathbf{x}_0$ with $ \| \nu \mathbf{w} + g(\mathbf{x}_0) \|_{\mathbf{x}_0}^* < \eta$, where $\eta < 1/4$ is the given neighborhood radius for initialization, and
$y_0$ is small enough, then not only is $\mathbf{t}-y_0\mathbf{w}\in\Ksc$, but the initial vector $(\mathbf{x}_0, y_0, \mathbf{s}_0 := \mathbf{t}-y_0\mathbf{w})$ is near the central path.

Since  in this application we may choose $\mathbf{w} \in \Ksc$ arbitrarily, we can initialize either variant of our IPA as follows: take any known $\mathbf{x}_0 \in \K^\circ$, and choose $\mathbf{w} := -g(\mathbf{x}_0)/\nu$. Now $\mathbf{x}_0$ is primal feasible by Eq.~\ref{prop:lhscb_gradient} and $\| \nu \mathbf{w} + g(\mathbf{x}_0) \|_{\mathbf{x}_0}^* = 0 < \eta$.

\begin{lemma} \label{lemma:y0_construction}
    Let $\mathbf{x}_0 \in \mathcal{K}^\circ$ with $\mathbf{x}_0^\top \mathbf{w} = 1$ be given such that $ \| \nu \mathbf{w} + g(\mathbf{x}_0) \|_{\mathbf{x}_0}^* < \eta$,
where $\eta < \frac{1}{4}$ is the radius of our desired neighborhood for initialization.
If
\[
y_0 \leq \mathbf{x}_0^\top \mathbf{t} - \nu \frac{\| \mathbf{t} - \mathbf{x}_0^\top \mathbf{t} \mathbf{w} \|_{\mathbf{x}_0}^*}{\eta - \| \nu \mathbf{w} + g(\mathbf{x}_0) \|_{\mathbf{x}_0}^*},
\]
$\mathbf{s}_0 = \mathbf{t} - y_0 \mathbf{w}$, and $\tau_0 := \frac{\mathbf{x}_0^\top \mathbf{s}_0}{\nu}$,
then $(\mathbf{x}_0, y_0, \mathbf{s}_0)$ is in the $\mathcal{N}(\eta, \tau_0)$ neighborhood of the central path of \eqref{eq:1constraint}, that is,
\begin{equation}
\label{eq:initNbd}
\| \mathbf{s}_0 + \tau_0 g(\mathbf{x}_0) \|_{\mathbf{x}_0}^* \leq \eta \tau_0,
\end{equation}
and $(\mathbf{x}_0, y_0, \mathbf{s}_0) \in \mathcal{F}^\circ$.
\end{lemma}

\begin{proof}
First, observe that  we can derive a lower bound on $\tau_0$ using the given  upper bound on $y_0$:
\begin{equation*}
\tau_0 = \frac{\mathbf{x}_0^\top \mathbf{s}_0}{\nu} = \frac{\mathbf{x}_0^\top (\mathbf{t} - y_0 \mathbf{w})}{\nu} = \frac{\mathbf{x}_0^\top \mathbf{t} - y_0 \mathbf{x}_0^\top \mathbf{w}}{\nu} = \frac{\mathbf{x}_0^\top \mathbf{t} - y_0}{\nu} \geq \frac{\| \mathbf{t} - \mathbf{x}_0^\top \mathbf{t} \mathbf{w} \|_{\mathbf{x}_0}^*}{\eta - \| \nu \mathbf{w} + g(\mathbf{x}_0) \|_{\mathbf{x}_0}^*}.
\end{equation*}
\newpage
Using the triangle inequality and this bound,
the inequality \eqref{eq:initNbd} follows:

\[
\begin{aligned}
\frac{1}{\tau_0} \| \mathbf{s}_0 + \tau_0 g(\mathbf{x}_0) \|_{\mathbf{x}_0}^* &= \frac{1}{\tau_0} \| \mathbf{t} - y_0 \mathbf{w} + \tau_0 g(\mathbf{x}_0) \|_{\mathbf{x}_0}^*\\
&= \frac{1}{\tau_0} \| \mathbf{t} - (\mathbf{x}_0^\top \mathbf{t} - \nu \tau_0) \mathbf{w} + \tau_0 g(\mathbf{x}_0) \|_{\mathbf{x}_0}^* \\
&\leq \frac{1}{\tau_0} \| \mathbf{t} - \mathbf{x}_0^\top \mathbf{t} \mathbf{w} \|_{\mathbf{x}_0}^* + \| \nu \mathbf{w} + g(\mathbf{x}_0) \|_{\mathbf{x}_0}^* \\
&\leq \frac{\eta - \| \nu \mathbf{w} + g(\mathbf{x}_0) \|_{\mathbf{x}_0}^*}{\| \mathbf{t} - \mathbf{x}_0^\top \mathbf{t} \mathbf{w} \|_{\mathbf{x}_0}^*} \| \mathbf{t} - \mathbf{x}_0^\top \mathbf{t} \mathbf{w} \|_{\mathbf{x}_0}^* + \| \nu \mathbf{w} + g(\mathbf{x}_0) \|_{\mathbf{x}_0}^* \leq \eta.
\end{aligned}
\]

This also means that $\frac{\mathbf{s}_0}{\tau_0} \in \mathcal{B}^*(-g(\mathbf{x}_0), \eta)$. As $-g(\mathbf{x}_0) \in (\mathcal{K}^*)^{\circ}$ and $\eta < 1$, it follows that $\mathbf{s}_0 \in (\mathcal{K}^*)^{\circ}$ holds. Therefore, $(\mathbf{x}_0, y_0, \mathbf{s}_0) \in \mathcal{F}^\circ$.
\end{proof}

\subsection{A two-phase approach for problems with a single equality constraint}
\label{sec:two-phase}

We now turn to a more general family of problems, which admits a simple two-phase strategy with strong complexity guarantees. Assume that our problem at hand is of the form \eqref{eq:1constraint}, as in the previous section, except that $\mathbf{w}$ is a given vector from $\Ksc$ that does not necessarily satisfy the norm inequality condition of Lemma \ref{lemma:y0_construction}. Also assume that we have a known strictly primal feasible solution $\vx_0$.

In this case, an initial primal-dual iterate near the central path can be found by applying Algorithm~\ref{alg:simple_short_step} to the following pair of auxiliary problems:
\begin{equation} \label{eq:phase_1}
    \left.
\begin{aligned}
\min \quad & \nu \mathbf{w}^\top \mathbf{x} \\
\st\quad\; & \mathbf{e}^\top \mathbf{x} = 1 \\
& \mathbf{x} \in \mathcal{K}
\end{aligned} \right\}
\hspace{2cm} \left.
\begin{aligned}
\max \quad & y \\
\st\quad\;\, & y \mathbf{e} + \mathbf{s} = \nu \mathbf{w} \\
& \mathbf{s} \in \mathcal{K}^*, \ y \in \mathbb{R}
\end{aligned} \right\},
\end{equation}
where $\mathbf{e} = -\frac{g(\mathbf{x}_0)}{\nu}$.

\begin{lemma}
\label{lemma:phase1theory1}
The problems \eqref{eq:phase_1} have strong duality with attained optimal solutions. The optimal value is positive, and there are strictly feasible solutions $(\vx,y,\vs)$ with $y=0$.
\end{lemma}
\begin{proof}
    The strong duality claim follows from our assumptions on $\vx_0$ and $\vw$: $\vx_0$ is a Slater point for the primal problem by definition, and $(0,\vs) = (0,\nu\vw)$ is a Slater point for the dual problem because $\vw\in\Ksc$. This also shows that $0$ is a suboptimal attainable value in the dual, therefore the optimal $y$ is positive.
\end{proof}
We can compute a strictly feasible solution $(\vx,0,\vs)$ to \eqref{eq:phase_1} as follows. By construction, \eqref{eq:phase_1} is an instance of \eqref{eq:1constraint} studied in the previous subsection, for which  $(\vx_0,0,-g(\mathbf{x}_0))$ is a strictly feasible solution on the central path, hence we can apply Lemma \ref{lemma:y0_construction} to initialize Algorithm \ref{alg:simple_short_step} and proceed to compute an approximate optimal solution to \eqref{eq:phase_1}, stopping at $y=0$ along the way.

The only (small, but technical) missing detail is that when we apply either of our methods to \eqref{eq:phase_1}, it is likely that we will never hit $y=0$ exactly; instead, the last full Newton step will take us from $y<0$ to $y>0$. We can easily fix this by taking a shorter Newton step in the last iteration: if $y<0$ but $y+\Delta y>0$, update $(\vx,y,\vs)$ to $(\vx+\alpha \Delta\vx,0,\vs+\alpha\Delta\vs)$ with the step size $\alpha$ that yields $y+\alpha\Delta y = 0$, and we expect that this will always work in practice. However, parts of our analysis do not apply if any of our Newton steps are damped. In the remainder of this subsection we show that if we solve the Phase 1 problem \eqref{eq:phase_1} by applying Algorithm \ref{alg:simple_short_step} with a sufficiently small neighborhood $\tilde{\eta}<\eta$, then even with a damped final step, the computed feasible solution $(\vx,0,\vs)$ yields a well-centered initial iterate for the solution of \eqref{eq:1constraint}. This claim is made precise by our next Lemma.
\begin{lemma}
\label{lemma:phase1-termination}
Assume that a strictly feasible solution to \eqref{eq:phase_1} of the form $(\vx,y,\vs)=(\vx,0,\vs)$ was computed with a modification of Algorithm \ref{alg:simple_short_step}, using a neighborhood radius of $\tilde{\eta}\leq 1/10$, where after a series of full Newton steps, the last iterate was arrived at with a damped Newton step of length $\alpha\in(0,1]$. Then $\vx/\mu$ (with $\mu=\vs^\T\vx/\nu$) is a primal feasible solution to \eqref{eq:1constraint} that also satisfies the inequality
\[\|\nu\vw + g(\vx/\mu)\|^*_{\vx/\mu} < 1/4.\]
\end{lemma}
\begin{proof}
Because our last step is damped, we can no longer assume that the algorithm returns a point within some $\mathcal{N}(\tilde\eta,\tau)$ neighborhood, but it can be shown that the final iterate is in some $\mathcal{N}(2\tilde\eta,\tau)$ neighborhood. (See Lemma \ref{lemma:damped_2eta} in Appendix \ref{sec:alg3_damped}.) That is, $(\vx,0,\vs)$ is a strictly feasible solution satisfying
\[ \|\vs + \tau g(\vx)\|^*_\vx \leq 2\tilde{\eta}\tau. \]
Using the properties of LHSCBs and substituting $\vs=\nu\vw$, we can equivalently write this inequality as
\[
\left\|\nu\vw + \frac{\tau}{\mu} g(\vx/\mu)\right\|^*_{\vx/\mu} \leq \frac{2\tilde{\eta}\tau}{\mu}.
\]
From this, we have the bound
\begin{equation}
\label{eq:phase1_bnd_tmp}
\|\nu\vw + g(\vx/\mu)\|^*_{\vx/\mu} \leq
\frac{2\tau}{\mu} \tilde{\eta} + \left|1-\frac{\tau}{\mu}\right| \|g(\vx/\mu)\|^*_{\vx/\mu} \stackrel{\ref{prop:lhscb_norm}}{=} \frac{2\tau}{\mu} \tilde{\eta} + \left|1-\frac{\tau}{\mu}\right|\sqrt{\nu},
\end{equation}
and we need to show that the right-hand side is below $1/4$ for our choice of $\tilde{\eta}$, using suitable bounds on $\tau/\mu$. Lemma \ref{lemma:DG} is \emph{not} applicable, as it assumes full Newton steps. An analogous argument, however, yields the following bounds for all step sizes (see Lemma \ref{lemma:alg3_mu_bounds} in Appendix \ref{sec:alg3_damped}):
\begin{equation}
\label{eq:mu-bounds-alpha}
\left( 1 - \frac{\tilde\eta^2}{\nu} \right) \tau \leq \mu \leq \frac{1}{1 - \vartheta} \tau.
\end{equation}

From \eqref{eq:mu-bounds-alpha} and $\nu\geq 1$, we get that for every $\tilde\eta<1/4$, $\tau/\mu \leq \frac{\nu}{\nu-\tilde{\eta}^2}$ and $|1-\tau/\mu| \leq \vartheta$. Thus, \eqref{eq:phase1_bnd_tmp} can be further bounded as
\begin{equation}
\|\nu\vw + g(\vx/\mu)\|^*_{\vx/\mu} \leq  2
\frac{\tilde{\eta}\nu}{\nu-\tilde{\eta}^2} + \vartheta\sqrt{\nu} =2\frac{\tilde{\eta}\nu}{\nu-\tilde{\eta}^2} + \frac{\tilde{\eta}/2}{\sqrt{\nu}+1}\sqrt{\nu},
\end{equation}
which in turn is strictly below $1/4$ for every $\nu\geq 1$ if $\tilde{\eta}\leq 1/10$, completing our proof.
\end{proof}

In summary, after obtaining a solution $(\vx,0,\vs)$ to our Phase 1 problem \eqref{eq:phase_1} as prescribed in Lemma \ref{lemma:phase1-termination}, we are ready to apply Lemma \ref{lemma:y0_construction} to initialize either of our algorithms using the computed $\vx/\mu$ and solve our original problem.

Aside from a single damped Newton step in its last iteration, Phase 1 is Algorithm \ref{alg:simple_short_step} applied to the problem \eqref{eq:phase_1}, meaning that the number of Phase 1 iterations is $\mathcal{O}(\sqrt{\nu})$. Therefore, the two-phase approach described in this section has a total iteration complexity of order $\mathcal{O}(\sqrt{\nu})$.

\subsection{Transforming a bounded problem with a known feasible solution to one with a single inhomogeneous equation}
\label{sec:transform_to_1constraint}

So far we have assumed that our primal problem (P) in standard form has only a single linear equality constraint. Next, we consider solving a general conic program in the form of \eqref{eq:PD_problem}, under two additional assumptions that ensure a finite infimum: suppose that we are given a strictly (primal) feasible vector $\vx_0$ and that the feasible region is \emph{a priori} bounded. In this section, we show that under these assumptions, every problem can be reduced to an instance of the problem studied in the previous section.

Consider an instance of our general conic problem \eqref{eq:PD_problem}, suppose that a vector $\vz\in\Ksc$ is readily available, perhaps in the form of $\vz = -g(\bar{\vx})$ for some $\bar{\vx}\in\K^\circ$, and that the feasible region of the primal problem, $\{\vx\in\Rn\,|\,A\vx=\mathbf\vb,\,\vx\in\K\}$, is \emph{a priori} known to be bounded explicitly. The latter implies that we have a redundant inequality $(0\leq)\vz^T\vx<U$ for some known upper bound $U$. Now consider the following representation of the feasible region:
\begin{equation}\label{eq:1constrain_origin}
    \begin{pmatrix}
        \vz^\T & 1 & 1\\
        \vz^\T & 1 & -U\\
        A      & \mathbf{0} & -\vb
    \end{pmatrix}
    \begin{pmatrix}
        \vx\\ \xi \\ \zeta
    \end{pmatrix}=
    \begin{pmatrix}
        U+1 \\ 0 \\ \mathbf{0}
    \end{pmatrix}, \qquad (\vx,\xi,\zeta) \in\K \times\R_+\times\R_+.
\end{equation}
On the one hand, it is clear that if $\vx$ is a feasible (resp., strictly feasible) solution to the original problem, then there exists a feasible (resp., strictly feasible) solution of the form $(\vx,\xi,1)$ of \eqref{eq:1constrain_origin}. (Choose $\xi:=U-\vz^\T\vx$.) The converse also holds: from the first two equations in \eqref{eq:1constrain_origin}, we see that every feasible solution to \eqref{eq:1constrain_origin} satisfies $\zeta=1$, which in turn implies $A\vx=\vb$ from the last equation. Therefore, the $\vx$ component of every $(\vx,\xi,\zeta)$ that satisfies \eqref{eq:1constrain_origin} is automatically a feasible solution to the original primal problem.

In this manner, every bounded primal problem can be rewritten in an equivalent form in which there is only a single inhomogeneous equality constraint. The homogeneous equations can then be incorporated into the cone constraint: defining
\[ \K_{A,\vb,U} := \left\{
(\vx,\xi,\zeta)\in\K\times\R_+\times\R_+
\,\middle|\,
    \begin{pmatrix}
        \vz^\T & 1 & -U\\
        A      & \mathbf{0} & -\vb
    \end{pmatrix}
    \begin{pmatrix}
        \vx\\ \xi \\ \zeta
    \end{pmatrix}=
    \begin{pmatrix}
        0 \\ \mathbf{0}
    \end{pmatrix}\right\},\]
we can write our original primal optimization problem as
\begin{equation}\label{eq:1constrain_end}
    \left. \begin{aligned}
    \min \ & \mathbf{c}^\top \vx \\
    \st\,\  & \vz^\T\mathbf{x} + \xi + \zeta = U+1 \\
    & (\vx,\xi,\zeta)\in\K_{A,\vb,U}
\end{aligned} \right\}.
\end{equation}
The reformulation \eqref{eq:1constrain_end} has two crucial properties: it has only a single, inhomogeneous, equality constraint besides the convex cone constraint, and the coefficient vector in that constraint belongs to the interior of the dual cone:
$(\vz,1,1) \in \Ksc\times\R_{++}\times\R_{++}\subseteq(\K_{A,\vb,U}^*)^\circ.$
Moreover, if we have an efficiently computable LHSCB for $\K$, say $f_\K$, then we also have one readily available for the new cone $\K_{A,\vb,U}$: a suitable LHSCB is simply the restriction of $f_\K + f_{\R_+^2}$ to the linear subspace in the definition of $\K_{A,\vb,U}$. Finally, any strictly feasible $\vx_0$ for our original primal problem extends automatically to a strictly feasible $(\vx,\xi,\zeta) = (\vx_0,1,U-\vz^\T\vx_0)$ of \eqref{eq:1constrain_end}.

This completes our transformation of our problem to an instance that can be solved using the two-phase method of the previous section.

\subsection{A direct two-phase approach for general equality constraints}

In this subsection, we show that another well-known initialization technique from the theory of IPAs can also be applied if a strictly feasible primal solution is available and the feasible set is bounded. This technique requires no transformations, additional variables, or modifications to the underlying cone.

For a suitably defined Phase 1 problem, we will follow the central path ``backwards'' by iteratively increasing the parameter $\tau$ until we get close enough to the analytic center of the feasible region, at which point a strictly feasible primal-dual solution can be defined in the $\mathcal{N}(\eta,\tau)$ neighborhood of the central path of our original problem \eqref{eq:PD_problem}.

For this initialization strategy, the Phase 1 problem can be defined as follows:
\begin{equation}
    \left. \begin{aligned}
        \min\;\;\, &-g(\mathbf{x}_0)^\T  \mathbf{x} \\
      \st\quad &   A \mathbf{x} = \mathbf{b} \\
         &\mathbf{x} \in \mathcal{K}
    \end{aligned} \right\} \qquad
    \left.
    \begin{aligned}
        \max \ & \mathbf{b}^\T \mathbf{y} \\
        \st\;\;\;&A^\T \mathbf{y} + \mathbf{s}  = -g(\mathbf{x}_0) \\
       & \mathbf{s} \in \mathcal{K}^*
    \end{aligned} \right\},
    \label{eq:phase1_backwardsIPA}
\end{equation}
where $\vx_0$ is the known strictly feasible primal solution, that is, $A \mathbf{x}_0 = \mathbf{b}$ and $\mathbf{x}_0 \in \mathcal{K}^\circ$. It is immediate that the triplet $(\vx,\vy,\vs) = (\vx_0,\mathbf{0}, -g(\vx_0))$
is a strictly feasible primal-dual solution to \eqref{eq:phase1_backwardsIPA}, moreover, this point is on the central path of this problem. The $\tilde\eta$-neighborhood of the central path (recall Eq. \eqref{eq:nbd}) is now defined as the set of strictly feasible points satisfying
 \begin{equation}
      \left\Vert -g(\mathbf{x}_0) - A^\T \mathbf{y} + \tau g(\mathbf{x} )\right\Vert_{\mathbf{x}}^* \leq \tilde\eta \tau.
     \label{eq:nbh_backwardsIPA}
 \end{equation}
This inequality obviously holds for our starting triplet with $\tau=1$ and every $\tilde\eta\geq0$.

 Thus, we may run Algorithm \ref{alg:simple_short_step} with two modifications: first, instead of decreasing $\tau$ in every iteration, we increase it to move ``backwards'' on the central path. The second change is in the stopping criterion: we increase $\tau$ until we find ourselves in the neighborhood of the central path of our \emph{original} problem; that is, until
 \[ \left\Vert \mathbf{c} - A^\T \mathbf{y} + \tau g(\mathbf{x}) \right\Vert_{\mathbf{x}}^* \leq{\eta} \tau \]
 is satisfied for some $\eta \leq 1/4$, meaning that we have a suitable starting point in the required neighborhood of the central path \eqref{eq:CP} of our original problem. Intuitively, Phase 1 terminates because all central paths meet at the analytic center; the precise argument is below.

 First, we need to show that the new iterates remain in the neighborhood of the Phase 1 central path, even if we increase the value of $\tau$ instead of decreasing it.
 \begin{lemma} \label{lemma:nbh_backwardsIPA}
    Let $(\mathbf{x}, \mathbf{y}, \mathbf{s}) \in \mathcal{N}(\tilde\eta , \tau)$, let $(\Delta \mathbf{x}, \Delta \mathbf{y}, \Delta \mathbf{s})$ be the solution of \eqref{eq:NS}, and let $(\mathbf{x}^+, \mathbf{y}^+, \mathbf{s}^+) = (\mathbf{x}, \mathbf{y}, \mathbf{s}) + (\Delta \mathbf{x}, \Delta \mathbf{y}, \Delta \mathbf{s})$. If $0 < \tilde\eta \leq \frac{1}{3}$ and $
\vartheta = \frac{\tilde\eta / 2}{\sqrt{\nu} + 1}$,
then the new point satisfies $(\mathbf{x}^+, \mathbf{y}^+, \mathbf{s}^+) \in \mathcal{N}(\tilde\eta, \tau^+)$ with $\tau^+ = (1 + \vartheta) \tau$.
\end{lemma}
The proof is essentially identical to the proof of Lemma~\ref{lemma:zplus-in-Nthetaplus}; we omit the details.

 As we assumed that the feasible set is bounded, the analytic $\vx_{\text{ac}}$ center of the feasible region exists, and the sequence of the vectors $\mathbf{x}$ generated by increasing $\tau$ and taking full Newton steps converges to a neighborhood of $\vx_{\text{ac}}$, inside of which the norm of the inverse of the barrier Hessian, $\|H(\vx)^{-1}\|$, is uniformly bounded for all $\vx$. Therefore, using the results of Lemma \ref{lemma:nbh_backwardsIPA},
 we can iterate our Phase 1 algorithm with a neighborhood radius $\tilde\eta<1/4$ until $\tau$ is large enough to satisfy $\frac{1}{\tau}||\mathbf{c}+g(\vx_0)||^*_\vx < \eta-\tilde\eta$, meaning that
 \[ \left\Vert -g(\vx_0) - A^\T \mathbf{y} + \tau g(\mathbf{x}) \right\Vert_{\mathbf{x}}^* \leq \tilde{\eta} \tau \implies \left\Vert \mathbf{c} - A^\T \mathbf{y} + \tau g(\mathbf{x}) \right\Vert_{\mathbf{x}}^* \leq 1/4 \tau,\]
 at which point we can switch to Phase 2, running either variant of our IPA with $\eta=1/4$ as before.

To describe the path switching more precisely: Phase 1 terminates after finite many steps, and if it ends at $(\tilde\vx,\tilde\vy,\tilde\vs)$, with a central path parameter $\tau$, then we can initialize Phase 2 with the triplet
\[(\vx,\vy,\vs) = (\tilde\vx,\tilde\vy,\mathbf{c}-A^\T\tilde\vy)\]
and the same value of $\tau$.
This triplet is automatically a  primal-dual strictly feasible solution,
since by assumption, $\left\Vert \frac{\mathbf{c} - A^\T\tilde\vy}{\tau} + g(\mathbf{x}) \right\Vert_{\mathbf{x}}^* \leq 1/4$,
and therefore $\vs=\mathbf{c} - A^\T\tilde\vy \in \Ksc$, and it is also in the $\mathcal{N}(\eta,\tau)$ neighborhood of the central path of the original problem, as required.

\subsection{Homogeneous self-dual embedding}
\label{sec:HSD}
In the most general case, when the feasibility and boundedness of primal and dual problems are unclear, we may apply a homogeneous self-dual embedding technique \cite{deKlerkRoosTerlaky1997, deKlerk1998, Luo2000}, in conjunction with either variant of our algorithm.

Assume that $\mathbf{\bar{x}} \in \K^\circ$, $\mathbf{\bar{s}} \in \Ksc$ and $\mathbf{\bar{y}} \in \mathbb{R}^m$ are given. (We denote these initial iterates with bars rather than the previously used $(\mathbf{x}_0,\mathbf{y}_0,\mathbf{s}_0)$ to emphasize that they are not feasible points, but arbitrary points in the interiors of their respective cones. We may choose, e.g., $\mathbf{\bar{y}}=\mathbf{0}$ and $\mathbf{\bar{s}}=-g(\mathbf{\bar{x}})$ with any $\mathbf{\bar{x}}\in\K^\circ$.)
Let $\bar{\mathbf{b}} = \mathbf{b} - A \mathbf{\bar{x}}$, $
\bar{\mathbf{c}} = \mathbf{c} - A^\top \bar{\mathbf{y}} - \bar{\mathbf{s}}$ and $\bar{z} = \mathbf{c}^\top \bar{\mathbf{x}} - \mathbf{b}^\top \bar{\mathbf{y}} + 1$, and let
\begin{equation}
\begin{aligned} \label{eq:HSD_mx}
	G =	\begin{pmatrix}
			\mathbf{0} & A & -\mathbf{b} & \phantom{-}\bar{\mathbf{b}} \\
			-A^{\top} & \mathbf{0} & \phantom{-}\mathbf{c} & -\bar{\mathbf{c}}  \\
			\phantom{-}\mathbf{b}^{\top} & -\mathbf{c}^{\top} & \phantom{-}0 & \phantom{-}\bar{z} \\
			-\mathbf{\bar{b}}^{\top} & \phantom{-}\mathbf{\bar{c}}^{\top} &  -\bar{z} & \phantom{-}0   \end{pmatrix}
	\end{aligned}.
\end{equation}
With this notation in place, the homogeneous self-dual (HSD) model can be formulated as follows:
\begin{equation} \left.
 \label{eq:HSD_model}
 \begin{aligned}
\min \ &(\mathbf{\bar{x}}^{\top} \mathbf{\bar{s}} + 1)\theta \\
\textrm{s.t.}\;\;&G \begin{pmatrix}
	\mathbf{y} \\ \mathbf{x} \\ \xi \\ \theta
\end{pmatrix} -
\begin{pmatrix}
	\mathbf{0} \\ \mathbf{s} \\ \kappa \\ 0
\end{pmatrix} =
\begin{pmatrix}
	\mathbf{0} \\ \mathbf{0} \\ 0 \\ -\mathbf{\bar{x}}^{\top} \mathbf{\bar{s}} - 1
\end{pmatrix} \\
 &            (\mathbf{x}, \xi) \in \mathcal{K} \times \mathbb{R}_{+}, \  ( \mathbf{s}, \kappa)\in \mathcal{K}^* \times \mathbb{R}_{+}, \  \mathbf{y} \in \mathbb{R}^m
\end{aligned}
 \right\}.
\end{equation}
The following theorem summarizes the well-known properties of the HSD model:

\begin{theorem}{\cite{deKlerkRoosTerlaky1997,deKlerk1998, Luo2000}} \label{th:HSD_prop}
The HSD model (\ref{eq:HSD_model}) is self-dual, and the point $
(\mathbf{x}, \mathbf{y}, \mathbf{s}, \xi, \theta, \kappa) = (\bar{\mathbf{x}}, \bar{\mathbf{y}}, \bar{\mathbf{s}}, 1, 1, 1)$ is a strictly feasible solution. The optimum value is $0$, and if $\xi > 0$ at an optimal solution, then $(\mathbf{x}/ \xi, \mathbf{y}/ \xi, \mathbf{s} / \xi)$ is an optimal solution to the original primal-dual pair of problems. If $\xi = 0$, the problem is either infeasible, unbounded, or the optimal duality gap is not zero.
\end{theorem}

The central path problem for \eqref{eq:HSD_model}, parametrized by the path parameter $\tau>0$, can be defined analogously to \eqref{eq:CP} as follows:
\begin{equation} \left.
\begin{aligned} \label{eq:HSD_CP}
G \begin{pmatrix}
	\mathbf{y} \\ \mathbf{x} \\ \xi \\ \theta
\end{pmatrix} -
\begin{pmatrix}
	\mathbf{0} \\ \mathbf{s} \\ \kappa \\ 0
\end{pmatrix} &=
\begin{pmatrix}
	\mathbf{0} \\ \mathbf{0} \\ 0 \\ -\mathbf{\bar{x}}^{\top} \mathbf{\bar{s}} - 1
\end{pmatrix} \\
 \mathbf{s} + \tau g(\mathbf{x})& = \mathbf{0} \\
 \kappa - \frac{\tau}{\xi} & = 0 \\
    (\mathbf{x}, \xi) \in \mathcal{K}^{\circ} \times \mathbb{R}_{++}, \  ( \mathbf{s}, \kappa)\in & (\mathcal{K}^*)^{\circ} \times \mathbb{R}_{++}, \  \mathbf{y} \in \mathbb{R}^m
\end{aligned} \right\}.
\end{equation}

It can be assumed without loss of generality that $\bar{\mathbf{s}}  = -g(\bar{\mathbf{x}})$. With this assumption, it is easy to show that the strictly feasible solution given in Theorem \ref{th:HSD_prop} is well-centered.
\begin{lemma}
 Assume that $\bar{\mathbf{s}}  = -g(\bar{\mathbf{x}})$ holds. Then the strictly feasible solution $
(\mathbf{x}, \mathbf{y}, \mathbf{s}, \xi, \theta, \kappa) = (\bar{\mathbf{x}}, \bar{\mathbf{y}}, \bar{\mathbf{s}}, 1, 1, 1)$  is on the central path \eqref{eq:HSD_CP}.
\end{lemma}
\begin{proof}
    The equations of \eqref{eq:HSD_CP} are satisfied for $(\bar{\mathbf{x}}, \bar{\mathbf{y}}, \bar{\mathbf{s}}, 1, 1, 1)$ and $\tau = 1$.
\end{proof}

As the problem (\ref{eq:HSD_model}) is self-dual, we apply Newton's method directly to this problem (rather than expliclitly formulating the dual instance with separate dual variables) to determine the search directions. The Newton system is the following:
\begin{equation} \left.
\begin{aligned} \label{eq:HSD_NS}
G
    \begin{pmatrix}
        \Delta \mathbf{y} \\ \Delta \mathbf{x} \\ \Delta \xi \\ \Delta \theta
    \end{pmatrix}
    -
    \begin{pmatrix}
        \mathbf{0} \\ \Delta \mathbf{s} \\ \Delta \kappa \\ 0
    \end{pmatrix} &=
\begin{pmatrix}
    \mathbf{0} \\ \mathbf{0} \\ 0 \\ 0
\end{pmatrix} \\
\tau H(\mathbf{x}) \Delta \mathbf{x} + \Delta \mathbf{s} &= - \mathbf{s} - \tau g (\mathbf{x}) \\
\frac{\tau}{\xi^2} \Delta \xi + \Delta \kappa &= - \kappa + \frac{\tau}{\xi}
\end{aligned} \right\}.
\end{equation}

Either variant of our IPA can be applied to solve \eqref{eq:HSD_model}, but instead of \eqref{eq:NS}, we need to solve \eqref{eq:HSD_NS} to determine the search directions.  The analysis can be adapted easily for the new Newton system. Introduce the new variables
\[ \tilde{\mathbf{x}} = \begin{pmatrix}
    \mathbf{x} \\ \xi
\end{pmatrix} \text{ and } \tilde{\mathbf{s}} = \begin{pmatrix}
    \mathbf{s} \\ \kappa
\end{pmatrix}\]
and their respective cones: $\tilde{\mathbf{x}} \in \tilde{\mathcal{K}} := \mathcal{K} \times \mathbb{R}_+$ and $\tilde{\mathbf{s}} \in \tilde{\mathcal{K}}^* := \mathcal{K}^* \times \mathbb{R}_+$. Given a $\nu$-LHSCB $f$ for $\K$, the extended cone $\tilde{\mathcal{K}}$ is the domain of the $(\nu+1)$-LHSCB
$\tilde{f}(\tilde{\mathbf{x}}) = f(\mathbf{x}) - \ln \xi$. We denote the gradient of $\tilde{f}$ by $\tilde{g}$.

It follows from the skew-symmetry of $G$

that $\Delta \mathbf{x}^{\top} \Delta \mathbf{s} + \Delta \xi \Delta \kappa = 0$, that is, $\Delta \tilde{\mathbf{x}}^{\top} \Delta \tilde{\mathbf{s}} = 0$, analogusly to the first claim of Lemma \ref{lemma:sd_properties}.
With this orthogonality result in place, the rest of the arguments in Section \ref{sec:full_step} leading up to Theorem \ref{theorem:compl1} can be repeated verbatim, with $\tilde{\mathbf{x}}, \tilde{\mathbf{s}}, \tilde{g}$, etc.~playing the roles of ${\mathbf{x}}, {\mathbf{s}}, g$, etc., to show that Algorithm \ref{alg:simple_short_step} with the Newton system \eqref{eq:HSD_NS} finds an $\varepsilon$-optimal solution to \eqref{eq:HSD_model} in $\mathcal{O}(\sqrt{\nu}\log \frac{\mathbf{\bar{x}}^\T\mathbf{\bar{s}}}{\varepsilon})$ iterations.

Algorithms \ref{alg:largest_update} and \ref{alg:real_largest_update} can be adapted similarly. In the definition of $\tau^+$, we need to use the new variables $\tilde{\mathbf{x}}$ and $\tilde{\mathbf{s}}$, and the local norm induced by the Hessian of $\tilde{f}(\tilde{\mathbf{x}})$.

It is important to note that in this manner we can only expect to get $\varepsilon$-optimal solutions to the HSD-model \eqref{eq:HSD_model}, that is, a strictly feasible solution to \eqref{eq:HSD_model} satisfying $\mathbf{x}^\T \mathbf{s} + \xi \kappa < \varepsilon$, which may not yield an $\hat\varepsilon$-optimal solution to the original primal-dual problem pair for an \emph{a priori} known $\hat\varepsilon$ even in the case when the original problem has optimal solutions, in the absence of any \emph{a priori} positive lower bound on the optimal value of the homogenizing parameter $\xi$.

Different stopping criteria for the HSD model \eqref{eq:HSD_model} were analyzed in several papers, see \cite{Freund2006, NaldiSinn2021, PolikTerlaky2009, Nesterov1999}, to name a few. In our implementation we have used the criteria proposed by Pólik and Terlaky \cite{PolikTerlaky2009}, which are theoretically sound stopping criteria based on a generalization of the ``Approximate Farkas Lemma'' of Todd and Ye \cite{ToddYe1998} to conic optimization. To be able to apply their results here, we need to prove that their assumption \cite[(3.1)]{PolikTerlaky2009}
\begin{equation}
\label{eq:PT_condition}
\kappa \xi \geq (1-\beta) \theta
\end{equation}
is satisfied for all iterates of our algorithms with a fixed constant $\beta \in (0,1)$.

\begin{lemma}
    Assume that $\bar{\mathbf{s}}  = -g(\bar{\mathbf{x}})$ holds. Then
    the P\'olik-Terlaky condition \eqref{eq:PT_condition} is satisfied  throughout all algorithm variants with $\beta := 1 - (1-\eta) \omega <1$, where
    $\omega = \min_k \frac{\tau_k}{\tau_{k-1}}$, where $\tau_k$ is the value of $\tau$ at the end of the $k^{\text{th}}$ iteration, and the minimum is taken over all iterations $k$.
\end{lemma}

\begin{proof}
    Let $(\vx,\vy,\vs,\xi,\theta,\kappa)$ be the iterate in the $k^{\text{th}}$ iteration (immediately after the Newton step) and $\mu=(\vx^\T\vs+\xi\kappa)/(\nu+1)$ be the corresponding normalized duality gap. As the iterates generated by the IPA remain in the neighborhood $\mathcal{N}(\eta, \tau_k)$ in the $k^{\text{th}}$ iteration, we have
    \[
    \begin{aligned}
        \eta^2 (\tau_k)^2 & \geq \| \tilde{ \mathbf{s} } + \tau_k \tilde{g} (\tilde{\mathbf{x}}) \|_{\tilde{\mathbf{x}}}^* = \\
        & = \begin{pmatrix}
            \mathbf{s} + \tau_k g(\mathbf{x}) \\
            \kappa - \tau_k / \xi
        \end{pmatrix}^{\top}
        \begin{pmatrix}
            H(\mathbf{x})^{-1} & 0 \\
            0 & \xi^2
        \end{pmatrix}
        \begin{pmatrix}
            \mathbf{s} + \tau_k g(\mathbf{x}) \\
            \kappa - \tau_k / \xi
        \end{pmatrix} = \\
        & = \left( \| \mathbf{s} + \tau_k g(\mathbf{x} ) \|_{\mathbf{x}}^* \right)^2 + \xi^2 (\kappa - \tau_k / \xi)^2.
    \end{aligned}
    \]
    Therefore, $| \kappa \xi - \tau_k | \leq \eta \tau_k$, implying $(1-\eta) \tau_k \leq \kappa \xi$.

    From \eqref{eq:HSD_model} and the skew-symmetry of $G$ it follows that $ \theta = \frac{\mathbf{x}^{\top} \mathbf{s} + \kappa \xi}{\bar{\mathbf{x}}^{\top} \bar{\mathbf{s}} +1}$. Using our assumption $\bar{\mathbf{s}}  = -g(\bar{\mathbf{x}})$ and \ref{prop:lhscb_lh}, we have $\bar{\mathbf{x}}^\T \bar{\mathbf{s}} = \nu$. Therefore, $\mu =  \theta$, that is, the value of the variable $\theta$ equals to the normalized duality gap $\mu$ in every iteration.

    Finally, from Lemma \ref{lemma:DG}, it follows that $\mu\leq\tau
    _{k-1}$ immediately after the Newton step (but before the step that updates $\tau$).

   Combining these bounds, we obtain the following bound in every iteration:
     \[ \kappa \xi \geq (1- \eta) \tau_k = (1 - \eta) \frac{\tau_k}{\tau_{k-1}} \tau_{k-1} \geq (1-\eta) \omega \mu = (1-\eta) \omega \theta = (1-\beta)\theta. \]
\end{proof}

\section{Application: Sums-of-squares lower bounds and dual nonnegativity certificates}
\label{sec:SOS-application}

We illustrate the efficiency and numerical stability of our algorithms in our motivating application from Section \ref{sec:intro}: computing moment-sums-of-squares lower bounds and exact sum-of-squares certificates for positive polynomials. We compare our IPAs applied directly to a nonsymmetric cone programming formulation of the problem to state-of-the-art semidefinite programming (SDP) solvers applied to an equivalent, but much larger, SDP formulation of the same problem.

Let us briefly recall the relevant theory: given a polynomial $f\in\R[\vx]$ and a basic, closed, and bounded semialgebraic set $\Delta$ of the form
\begin{equation}
\label{eq:Delta-def}
    \Delta = \left\{\mathbf{x} \in\R^n\,|\, g_i(\vx)\geq 0, i=1,\dots,m\right\},
\end{equation}
with $g_i\in\R[\vx]$, the minimum of $f$ over $\Delta$ is clearly certified to be at least $\gamma$ if $f-\gamma$ can be written as follows:
\begin{equation}
\label{eq:fminusgamma-SOS}
f(\vx) - \gamma = \sigma_0(\vx)+\sum_{i=1}^m g_i(\vx)\sigma_i(\vx),
\end{equation}
where each $\sigma_i\in\R[\vx]$ is a sum of perfect squares.
Under an additional assumption on the representation \eqref{eq:Delta-def} of $\Delta$, called the \emph{Archimedean property} \cite[Sec. 3.4.4]{BlekhermanParriloThomas2013}, the following onverse also holds. Let $SOS_d$ be the (finite dimensional) cone of polynomials of degree at most $d$ that can be written as a sum of perfect squares, define the degree-$d$ \emph{truncated quadratic module} as
\[ \mathcal{M}_d(\mathbf{g}) := \left\{\sigma_0(\vx)+\sum_{i=1}^m g_i(\vx)\sigma_i(\vx)\,\middle|\, \sigma_0 \in SOS_d,\, \sigma_i \in SOS_{d-\deg g_i} (i=1,\dots,m) \right\}, \]
and let
\begin{equation}
\label{eq:gammadstar-def}
\gamma^*_d := \sup\{\gamma\,|\,f-\gamma\in \mathcal{M}_d(\mathbf{g})\}.
\end{equation}
Putinar's Positivstellensatz \cite{Putinar1993} states that
\[
 \lim_{d\to\infty}\gamma^*_d =  \min_{\vx\in\Delta}f(\vx).
\]
This provides a recipe for computing a convergent series of lower bounds (called \emph{SOS lower bounds}) on the global minimum of the (not necessarily convex) polynomial $f$ by solving a series of increasingly large convex optimization problems (known as the \emph{SOS hierarchy}).

Under the Archimedean assumption, the convergence of the sequence $(\gamma^*_d)_{d=1,2,\dots}$ is ``generically'' finite in a sense precisely defined in \cite{Nie2014}, but might be infinite even if $f$ is univariate \cite{Stengle1996}, and in practice, the degree necessary to certify a sharp bound may be much higher than the degrees of $f$ and $g_i$ \cite{BaldiMourrain2023}. On the other hand, using high-degree polynomials $\sigma_i$ in the representation \eqref{eq:fminusgamma-SOS} may lead to numerical problems; this can be mitigated by representing the polynomials in an interpolant basis instead of the more commonly employed monomial basis \cite{Papp2017}.

The set $\mathcal{M}_d(\vg)$ is by definition a finite-dimensional convex cone, but optimization over it is not obviously tractable. In particular, no tractable LHSCB is known for this set. It is, however, a projected spectrahedron \cite[Chapter 6]{BlekhermanParriloThomas2013}, which means that linear optimization problems over it (such as the problem \eqref{eq:gammadstar-def} defining $\gamma_d^*$) can be formulated as SDPs \cite{Parrilo2000}, but only at the cost of substantially increasing the number of optimization variables. For example, in the univariate case, $SOS_d$ is a $(d+1)$-dimensional cone, but its semidefinite representation involves $(d+1)\times(d+1)$ real symmetric matrices, effectively squaring the number of optimization variables. As another illustration, consider the case when $m$ is large: the value of $m$ has no bearing on the dimension of $\mathcal{M}_d(\vg)$, but the SDP representation grows linearly with $m$.

Despite the lack of tractable LHSCBs, we can employ certain nonsymmetric conic IPAs to compute SOS lower bounds and optimality certificates, because tractable LHSCBs are known for the dual cone $\mathcal{M}_d(\vg)^*$; see, e.g., \cite{PappYildiz2019}. The sequence of lower bounds obtained by optimizing over $\mathcal{M}_d(\vg)^*$, particularly in connection with SDP, is also referred to as Lasserre's \emph{moment-SOS hiearchy} \cite{Lasserre2001}. To summarize this approach, we need some additional notation. Let $\R_{n,d}$ be the space of $n$-variate polynomials of degree at most $d$. The dual cone of  $\mathcal{M}_d(\vg)$, sometimes called the \emph{pseudo-moment cone}, is defined
as
\[
\mathcal{M}_d(\vg)^*=\left\{ \ell\in\mathbb{R}_{n,d}^{*}\,\middle|\,\;\ell(\sigma)\geq0\,\,\forall\sigma\in\mathcal{M}_d(\vg)\right\},
\]
where $\R_{n,d}^*$ is the space of
linear functionals on $\R_{n,d}$. (Since we are working with finite dimensional spaces, all linear functionals are automatically bounded.) If $\mathcal{M}_{d}(\vg)$ is a proper cone, then by strong duality, $\gamma_d^*$ is also the optimal value of the dual problem of \eqref{eq:gammadstar-def}:
\begin{equation}
\label{eq:gammadstar-dualdef}
\inf \{ \ell(f)\,|\, \ell(1) = 1,\; \ell\in\mathcal{M}_d(\vg)^* \}.
\end{equation}

Because our cones are finite dimensional, the polynomials $f$ and $1$ can be represented by their coefficient vectors $\mathbf{f}\in\R^N$ and $\mathbf{1}\in\R^N$, respectively, in a fixed (ordered) basis; similarly, each $\ell\in\mathcal{M}_d(\vg)^*$ can be represented by a vector $\bm\lambda\in\R^N$ describing its action on the same ordered basis. In this manner, the dual definition \eqref{eq:gammadstar-dualdef} of $\gamma_d^*$ can be written as follows:
\begin{equation}\label{eq:SOS-lb-dual}
    \left. \begin{aligned}
    \min \ & \mathbf{f}^\T \bm\lambda \\
    \st\,\ & \mathbf{1}^\T\bm\lambda = 1 \\
    & \bm\lambda\in\mathcal{M}_{d}(\vg)^*
\end{aligned} \right\}.
\end{equation}
Conventionally, this problem would be solved as an SDP, using the semidefinite representation of $\mathcal{M}_d(\vg)^*$, a well-known spectrahedron \cite{Lasserre2001,Laurent2008}. However, \eqref{eq:SOS-lb-dual} is also an optimization problem of the form studied in Section \ref{sec:two-phase}: a standard form conic optimization problem with a single equality constraint, whose coefficient vector $\mathbf{1}$ belongs to $\mathcal{M}_d(\vg)^\circ$ for every $d$. Thus, we can circumvent the SDP and compute $\gamma_d^*$ with our IPAs instead, solving \eqref{eq:SOS-lb-dual} directly as written. And as we mentioned in Section \ref{sec:intro}, the numerical vectors $\bm\lambda$ computed with our IPAs are \emph{dual nonnegativity certificates} in the sense of \cite{DavisPapp2022, DavisPapp2024}, meaning that an exact, rational sums-of-squares representation of $f-\hat\gamma_d$ in the form \eqref{eq:fminusgamma-SOS} can be computed \emph{in closed-form} from the computed numerical approximation $\hat\gamma_d$ of $\gamma_d^*$ and the computed near-optimal $\bm\lambda$. The details of this certification process are beyond the scope of this short summary, but it suffices to say here that by \cite[Thm.~2.4]{DavisPapp2022}, if $(\bm\lambda,\gamma,\mathbf{f}-\gamma\mathbf{1})\in\mathcal{N}(1,\tau)$ (for any $\tau>0$), then $\bm\lambda$ is a certificate proving the nonnegativity of $f-\gamma$ (without having to construct an explicit SOS representation to verify this fact). Therefore, we may stop our IPAs at any time, and the current iterate $\bm\lambda$ can be used to certify the current bound, as long as the numerical errors remain small enough that the iterates (which theoretically stay within $\mathcal{N}(1/4,\tau)$) remain at least in $\mathcal{N}(1,\tau)$. We underscore that the same would not be true if we used a different neighborhood definition in our IPAs.

We shall illustrate all of the above using a seemingly innocuous, but surprisingly challenging instance.
\begin{example}[\cite{Stengle1996,BaldiSlot2024}]
\label{ex:challenge1D}
Let $m=n=1$, $f(x)=1-x^2$, and $g_1(x)=(1-x^2)^3$. Then $\Delta=[-1,1]$ and $\min_{x\in\Delta}f(x) = 0$, the optimizers are $x=\pm 1$. The quadratic module associated with $\{g_1\}$ is Archimedean:
\[(1 - x^2)^2 + \left(x(-2 + x^2)\right)^2 + (1 - x^2)^3 = 2-x^2 \in \mathcal{M}_6(\vg).\]
The exact bounds $\gamma_d^*$ from the Lasserre hierarchy are not known for this instance, but a simple degree-parity argument shows that this problem does not have finite convergence, meaning that altough $\lim_{d\to\infty}\gamma_d^* = 0$, the value of every $\gamma_d^*$ is strictly negative.

More involved analysis reveals that to reach $\gamma_d^* \geq -\varepsilon$, we need a degree $d$ of order $\Omega(\varepsilon^{-1/2})$ at least \cite[Thm.~4]{Stengle1996}. Below, we present numerical evidence for a new conjecture: the exact values of the SOS lower bounds in this instance are given by $\gamma_{2d}^* = -1/(d(d-2))$.
\end{example}

Table \ref{tbl:challenge1D-bounds} shows numerical results comparing SOS lower bounds computed by solving \eqref{eq:SOS-lb-dual} using our IPAs, along with the solutions of the equivalent SDP formulations computed with SeDuMi and Mosek. To improve the conditioning of the formulations, all polynomials were represented as Chebyshev interpolants, both in the nonsymmetric formulations and in the equivalent SDPs, as previously proposed in, e.g., \cite{LofbergParrilo2004, Papp2017}.

Both IPAs were implemented with the largest update strategy (Algorithm \ref{alg:real_largest_update}), using line search. Since \eqref{eq:SOS-lb-dual} has a single equality constraint with the coefficient vector in the interior of the dual cone, the two-phase method proposed in Subsection \ref{sec:two-phase} can be applied to
determine a well-centered and strictly feasible starting point for the feasible IPA. For Example \ref{ex:challenge1D}, this approach was highly effective and the first phase always produced a suitable starting point in 13--17 iterations, virtually independent of the problem size.

To avoid constructing a strictly feasible starting point, we also implemented the self-dual embedding technique. As shown by our numerical results, the feasible variant was more effective in both running time and the quality of the obtained lower bounds.

The experiments were carried out on a MacBook Pro with an Apple M3 CPU and 24 GB of unified memory. The algorithms were implemented in MATLAB R2024b. We formulated the SDP models using YALMIP version R20230622 \cite{YALMIP}, and solved with Mosek version 10.2.7 \cite{mosek10} and SeDuMi 1.3.7. \cite{sedumi}. The timing results refer to solver times (excluding the times YALMIP took to process the models and call the solvers).

\paragraph{Results} SeDuMi reports numerical errors from $d=60$, and although Mosek and our IPAs all claim to have found near-optimal, near-feasible solutions, there is a large disagreement between the bounds from $d=80$ onwards. As detailed below, the reported ``optimal'' values from both SDP solvers are incorrect from $d \geq 80$; in contrast, the feasible 2-phase IPA and its HSD variant both yield high-quality bounds up to approximately $d=600$.

Note that in Table \ref{tbl:challenge1D-bounds} we reported \emph{dual} objective values, as these are supposed to be lower bounds on the optimal values of \eqref{eq:gammadstar-def} and \eqref{eq:SOS-lb-dual}, that is, they ought to be valid global lower bounds on $\gamma_d^*$, and thus on $f$.

It is easy to verify (rigorously, in rational arithmetic) that the computed primal solutions $\bm\lambda$ are indeed strictly feasible solutions, after rescaling them to satisfy the only equality constraint exactly. This also yields a certified, rational \emph{upper} bound on the true value of $\gamma_d^*$. For example, in the $d=80$ case, it can be certified using the solution returned by our 2-phase feasible IPA that $-1/\gamma_d^*  \leq 1520.000032$. This confirms that the SDP lower bounds cannot be sharp.

\begin{table}[]
\resizebox{\textwidth}{!}{
\begin{tabular}{@{}crrrrrrrr@{}}
\toprule
\textbf{} & \multicolumn{2}{c}{\textbf{2-phase IPA}} & \multicolumn{2}{c}{\textbf{IPA with HSD}} & \multicolumn{2}{c}{\textbf{Mosek}} & \multicolumn{2}{c}{\textbf{SeDuMi}} \\ \midrule
\textbf{d} & \multicolumn{1}{c}{\textbf{dObj}} & \multicolumn{1}{c}{\textbf{-1/dObj}} & \multicolumn{1}{c}{\textbf{dObj}} & \multicolumn{1}{c}{\textbf{-1/dObj}} & \multicolumn{1}{c}{\textbf{dObj}} & \multicolumn{1}{c}{\textbf{-1/dObj}} & \multicolumn{1}{c}{\textbf{dObj}} & \multicolumn{1}{c}{\textbf{-1/dObj}} \\ \midrule
\textbf{20} & -1.250000E-02 & 79.999838 & -1.250000E-02 & 80.000021 & -1.250001E-02 & 79.999954 & -1.250001E-02 & 79.999964 \\
\textbf{40} & -2.777800E-03 & 359.998865 & -2.777780E-03 & 360.000020 & -2.777784E-03 & 359.999181 & -2.777830E-03 & 359.993258 \\
\textbf{60} & -1.190500E-03 & 839.998546 & -1.190480E-03 & 840.000042 & -1.190548E-03 & 839.949616 & -1.190552E-03 & 839.946371 \\
\textbf{80} & -6.580000E-04 & 1,519.998800 & -6.578950E-04 & 1,520.000021 & -7.230688E-04 & 1,382.994260 & -6.825633E-04 & 1,465.065584 \\
\textbf{100} & -4.170000E-04 & 2,399.997850 & -4.166670E-04 & 2,399.999972 & -4.385459E-04 & 2,280.263024 & -4.558204E-04 & 2,193.846524 \\
\textbf{120} & -2.870000E-04 & 3,479.990060 & -2.873560E-04 & 3,479.999717 & -4.202670E-04 & 2,379.439737 & -3.166395E-04 & 3,158.165674 \\
\textbf{140} & -2.100000E-04 & 4,759.984770 & -2.100840E-04 & 4,759.999279 & -4.115541E-04 & 2,429.814209 & -3.682441E-04 & 2,715.590012 \\
\textbf{160} & -1.600000E-04 & 6,239.979070 & -1.602560E-04 & 6,239.998055 & -3.657808E-04 & 2,733.877776 & -1.711278E-04 & 5,843.585905 \\
\textbf{180} & -1.260000E-04 & 7,919.962130 & -1.262630E-04 & 7,919.996061 & -3.366643E-04 & 2,970.317910 & -1.429270E-04 & 6,996.578673 \\
\textbf{200} & -1.020000E-04 & 9,799.935480 & -1.020410E-04 & 9,799.993284 & -3.508363E-04 & 2,850.332192 & -1.156350E-04 & 8,647.900722 \\
\textbf{400} & -2.530000E-05 & 39,599.555200 & -2.525260E-05 & 39,599.827255 & -1.518160E-04 & 6,586.921010 & -6.482170E-05 & 15,426.932648 \\
\textbf{600} & -1.120000E-05 & 89,398.411500 & -1.118580E-05 & 89,399.223922 & -7.604590E-05 & 13,149.952857 & -7.689500E-06 & 130,047.467326 \\
\textbf{800} & -6.280000E-06 & 159,191.174000 & -6.281400E-06 & 159,200.068639 & -4.913750E-05 & 20,351.055711 & 1.306050E-05 & -76566.747062 \\
\textbf{1000} & -4.020000E-06 & 248,972.289000 & -4.016060E-06 & 249,000.230609 & -2.195120E-05 & 45,555.596049 & 9.723700E-06 & -102841.510947 \\
\textbf{1200} & -2.790000E-06 & 358,686.431000 & -2.787040E-06 & 358,803.168207 & \multicolumn{1}{c}{--} & \multicolumn{1}{c}{--} & 6.180700E-06 & -161793.971557 \\
\textbf{1400} & -2.050000E-06 & 488,307.483000 & -2.046640E-06 & 488,604.678218 & \multicolumn{1}{c}{--} & \multicolumn{1}{c}{--} & 2.964000E-06 & -337381.916329 \\
\textbf{1600} & -1.570000E-06 & 637,951.582000 & -1.566220E-06 & 638,478.164902 & \multicolumn{1}{c}{--} & \multicolumn{1}{c}{--} & \multicolumn{1}{c}{--} & \multicolumn{1}{c}{--} \\
\textbf{1800} & -1.240000E-06 & 807,073.298000 & -1.236920E-06 & 808,458.887837 & \multicolumn{1}{c}{--} & \multicolumn{1}{c}{--} & \multicolumn{1}{c}{--} & \multicolumn{1}{c}{--} \\
\textbf{2000} & -1.000000E-06 & 996,751.362000 & -1.001130E-06 & 998,874.189340 & \multicolumn{1}{c}{--} & \multicolumn{1}{c}{--} & \multicolumn{1}{c}{--} & \multicolumn{1}{c}{--} \\ \bottomrule
\end{tabular}
}
\caption{Near-optimal dual objective function values (numerical moment-SOS lower bounds) for the polynomial optimization problem in Example \ref{ex:challenge1D} in support of our conjecture that \mbox{$\gamma_{2d}^*=-1/(d(d-2))$}. The ``2-phase IPA'' method is Alg.~\ref{alg:real_largest_update} with the two-phase variant outlined in Sec.~\ref{sec:two-phase}; ``IPA with HSD'' is Alg.~\ref{alg:real_largest_update} in conjunction with the self-dual embedding of Sec.~\ref{sec:HSD}. The Mosek and SeDuMi columns represent solutions computed using the standard SOS SDP formulation (representing all polynomials in the Chebyshev interpolant basis). Missing values indicate insufficient memory. Both IPAs show excellent agreement with our conjecture up to approximately $d=600$ and reasonable agreement up to $d=2000$; Mosek and SeDuMi are incorrect from $d=80$ onwards.}
\label{tbl:challenge1D-bounds}
\end{table}

In Table \ref{tbl:challenge1D-times}, we compare the running times of each of the solvers. Although the incorrect bounds make the comparisons somewhat moot, it is worthwhile to point out that solving the nonsymmetric conic formulations of these problems scales better than their SDP counterparts, underscoring another key motivation behind the development of nonsymmetric conic solvers \cite{CoeyKapelevichVielma2022,PappYildiz2022}.

\begin{table}[]
\resizebox{\textwidth}{!}{
\begin{tabular}{@{}crrrrrrrr@{}}
\toprule
\textbf{} & \multicolumn{2}{c}{\textbf{2-phase IPA}} & \multicolumn{2}{c}{\textbf{IPA with HSD}} & \multicolumn{2}{c}{\textbf{Mosek}} & \multicolumn{2}{c}{\textbf{SeDuMi}} \\ \midrule
\textbf{d} & \multicolumn{1}{c}{\textbf{Iterations}} & \multicolumn{1}{c}{\textbf{Time}} & \multicolumn{1}{c}{\textbf{Iterations}} & \multicolumn{1}{c}{\textbf{Time}} & \multicolumn{1}{c}{\textbf{Iterations}} & \multicolumn{1}{c}{\textbf{Time}} & \multicolumn{1}{c}{\textbf{Iterations}} & \multicolumn{1}{c}{\textbf{Time}} \\ \midrule
\textbf{20} & 85 & 0.19 & 74 & 0.12 & 14 & 0.01 & 21 & 0.48 \\
\textbf{40} & 126 & 0.28 & 134 & 0.19 & 20 & 0.02 & 28 & 0.23 \\
\textbf{60} & 149 & 0.42 & 146 & 0.30 & 20 & 0.05 & 33 & 0.39 \\
\textbf{80} & 160 & 0.49 & 151 & 0.56 & 21 & 0.13 & 31 & 0.43 \\
\textbf{100} & 170 & 0.62 & 153 & 0.52 & 28 & 0.33 & 32 & 0.93 \\
\textbf{120} & 163 & 0.82 & 169 & 0.61 & 28 & 0.54 & 35 & 2.26 \\
\textbf{140} & 164 & 0.87 & 175 & 0.75 & 23 & 0.76 & 32 & 3.14 \\
\textbf{160} & 170 & 1.08 & 188 & 1.00 & 28 & 1.36 & 41 & 6.33 \\
\textbf{180} & 171 & 1.92 & 191 & 1.37 & 22 & 1.71 & 41 & 10.58 \\
\textbf{200} & 171 & 2.19 & 209 & 2.02 & 25 & 2.30 & 41 & 14.65 \\
\textbf{400} & 314 & 15.19 & 349 & 16.95 & 16 & 16.34 & 32 & 140.51 \\
\textbf{600} & 353 & 36.84 & 383 & 41.19 & 24 & 108.39 & 31 & 664.20 \\
\textbf{800} & 357 & 64.71 & 405 & 127.67 & 22 & 324.22 & 37 & 2398.18 \\
\textbf{1000} & 397 & 118.35 & 465 & 225.34 & 24 & 869.12 & 46 & 7203.89 \\
\textbf{1200} & 422 & 192.70 & 490 & 335.51 & -- & -- & 42 & 13669.90 \\
\textbf{1400} & 436 & 302.84 & 527 & 532.86 & -- & -- & 40 & 24508.19 \\
\textbf{1600} & 442 & 447.44 & 505 & 732.90 & -- & -- & -- & -- \\
\textbf{1800} & 406 & 527.35 & 515 & 867.49 & -- & -- & -- & -- \\
\textbf{2000} & 464 & 815.18 & 499 & 1,004.33 & -- & -- & -- & -- \\ \bottomrule
\end{tabular}
}
\caption{Total running times (in seconds) and numbers of iteration for the computations presented in Table \ref{tbl:challenge1D-bounds}. For the 2-phase IPA, the iterations of both phases are included. Our IPAs were run until failure due to numerical errors, resulting in a higher than usual number of iterations. Regardless, for larger values of $d$, the time and memory complexity of the nonsymmetric IPAs scale better than SDP, matching the theoretical predictions.}
\label{tbl:challenge1D-times}
\end{table}

\section{Concluding remarks}

Most convex cones $\K$ whose membership problem is not known to be NP-hard have the property that either $\K$ or its dual $\K^*$ has a known LHSCB whose gradient and Hessian can be computed efficiently. This means that the algorithms presented in this paper are among the most generally applicable second-order methods for conic optimization. This does not come at a price of inefficiency either, as the iteration complexity of the methods matches those of the commonly used IPAs for symmetric cone programming.

The HSD variant of the method is the most general one, applicable to any problem whose primal or dual feasibility status is unknown. We opted for the less commonly used extended form of HSD, as it is more directly adaptable to feasible methods and comes with theoretically sound and quantitative stopping criteria for the detection of large optimal solutions. A drawback of using HSD is that its complexity can only be described in terms of the achieved tolerance for the embedded model, but in the absence of any \emph{a priori} positive lower bound on the optimal value of the homogenizing parameter $\xi$, a near-feasible and near-optimal solution to the embedded model may not yield a similarly near-feasible and near-optimal solution to the original problem even in the ideal case, when the original problem has an optimal solution of small norm.

In contrast, the initialization approaches studied in Sections \ref{sec:dual-membership-init} and \ref{sec:two-phase} come with polynomial-time performance guarantees, and yield near-optimal and very nearly feasible solutions. In our computations, the number of Phase 1 iterations was practically constant, which means that initialization was not a significant overhead. In addition, in our experiments, we found the feasible two-phase method to be numerically more reliable than HSD.

More detailed comparisons of the different variants of these IPAs, practical performance enhancements, and numerical comparisons to other nonsymmetric cone programming algorithms are the subject of a separate study.

\bibliographystyle{siamplain}
\bibliography{nonsymmetric}

\appendix

\section{Self-concordant barrier functions and local norms}
\label{sec:lhscb}
The notion of self-concordant barriers, originally introduced by Nesterov and Nemirovski \cite[Chapter 2]{Nesterov1994}, are quite broadly known at least to experts of interior-point algorithms, but for completeness and for the ease of referencing a few specific properties used throughout the paper, we have collected a few key results about them in this Appendix. For a more in-depth study, the reader is recommended to confer the excellent monographs \cite{Renegar2001} and \cite[Chapter 5]{Nesterov2018}.

\begin{definition}[Self-concordant barrier]
Let $\K\subseteq\Rn$ be a convex cone with a non-empty interior $\K^\circ$. A function $f:\K^\circ\to\R$ is called a \emph{strongly nondegenerate, self-concordant barrier} for $\K$ if $f$ is strictly convex, at least three times differentiable, and has the following two properties (with $\operatorname{bd}(\K)$ denoting the boundary of $\K$, and $g$ and $H$ denoting the gradient and Hessian of $f$):
\begin{enumerate}[\hspace{1em} 1.]
\item $f(\vx)\to\infty$  as $\vx\to\operatorname{bd}(\K)$,
\item $|D^3 f(\vx)[\vh,\vh,\vh]| \leq 2\left(D^2f(\vx)[\vh,\vh]\right)^{3/2}$,
\item $\nu := \sup_{\vx\in\K^\circ} g(\vx)^\T H(\vx)^{-1}g(\vx) < \infty.$
\end{enumerate}
We refer to the parameter $\nu$ above as the \emph{barrier parameter} of $f$; we also use the shorthand \emph{$\nu$-SCB} to mean ``an SCB with barrier parameter $\nu$''.
\end{definition}

\begin{definition}[Logarithmic homogeneity and LHSCBs]
    Let $f: \ \mathcal{K}^{\circ} \rightarrow \mathbb{R}$ be a $\nu$-SCB function for the proper convex cone $ \mathcal{K}$.  We say that $f$ is a \emph{logarithmically homogeneous self-concordant barrier (LHSCB)} with barrier parameter $\nu$ for $\mathcal{K}$, or a \emph{$\nu$-LHSCB} for short, if it satisfies the \emph{logarithmic homogeneity property}
    \[ f(t \mathbf{x}) = f(\mathbf{x}) - \nu \ln (t) \quad \text{for all } \mathbf{x} \in \mathcal{K}^{\circ} \text{ and } t > 0.\]
\end{definition}

We associate with each $\vx\in\K^\circ$ the \emph{local inner product} $\langle\cdot,\cdot\rangle_\vx : \K \times \K \to \R$ defined as $\langle \vu,\vv\rangle_\vx := \vu^\T H(\vx) \vv$ and the \emph{local norm} $\|\cdot\|_\vx$ induced by this local inner product. Thus, $\|\vv\|_\vx = \|H(\vx)^{1/2}\vv\|$. We define the local (open) ball centered at $\vx$ with radius $r$ by $\mathcal{B}(\vx,r) := \{\vv\in\R^n\,|\, \|\vv-\vx\|_{\vx} < r\}$. When $r=1$, these balls are also called the \emph{Dikin ellipsoids}. Crucially, for every $\vx\in\K^\circ$, we have $\mathcal{B}(\vx,1)\subset\K^\circ$.

Analogously, we define the \emph{dual local inner product} $\langle \cdot,\cdot \rangle_\vx^*: \K^*\times\K^*\to\R$ by  $\langle \vs, \vt \rangle_\vx^* := \vs^\T H(\vx)^{-1}\vt$, its induced \emph{dual local norm} $\|\cdot\|_{\vx}^*$, and the corresponding open balls $\mathcal{B}^*(\vx,r)$. The dual local norm, by definition, satisfies the identity $\|\vt\|_{\vx}^* = \|H(\vx)^{-1/2}\vt\|$.

The most important properties of SCBs and LHSCBs applied in the analysis of our algorithms are compiled in the next two Propositions.

\begin{proposition}[\protect{\cite[Chapters 2--3]{Renegar2001}, \cite[Chapter 5]{Nesterov2018}, and \cite[Lemma 1]{Nesterov2012}}]
    Let $f$ be a $\nu$-SCB for the proper convex cone $\mathcal{K}$, let $\vx\in\K^\circ$, and let $\mathbf{n} (\mathbf{x} ) = - H (\mathbf{x})^{-1} g (\mathbf{x})$ be the (unconstrained) Newton step for $f$ at $\mathbf{x}$.
       \setlist[enumerate,1]{label=\textup{(A\arabic*)}, leftmargin=5em, labelsep=0.5em, itemindent=0pt}
    \begin{enumerate}
        \item \label{prop:scb_NS}
        If $\|\mathbf{n}(\mathbf{x})\|_{\mathbf{x}} < 1$, then $\mathbf{x}^+ := \mathbf{x} + \mathbf{n} (\mathbf{x})$ satisfies
\[
\|\mathbf{n} (\mathbf{x}^+)\|_{\mathbf{x}^+} \leq \left( \frac{\|\mathbf{n}(\mathbf{x})\|_{\mathbf{x}}}{1 - \|\mathbf{n}(\mathbf{x})\|_{\mathbf{x}}} \right)^2.
\]
    \item \label{prop:scb_norm1}
   If $\mathbf{u} \in \mathcal{B}_{\mathbf{x}}(\mathbf{x},1)$, then $\vu\in\K^\circ$, and for all $\mathbf{v} \neq \mathbf{0}$, one has
\[
\frac{\|\mathbf{v}\|_{\mathbf{u}}^*}{\|\mathbf{v}\|_{\mathbf{x}}^*} \leq \frac{1}{1 - \|\mathbf{u} - \mathbf{x}\|_{\mathbf{x}}}.
\]
\item \label{prop:scb_norm2} Let $\mathbf{x}, \mathbf{u} \in \mathcal{K}^\circ$, and let $r:= \| \mathbf{x} - \mathbf{u} \|_\mathbf{u} < 1$. Then
\[ \| g(\mathbf{x}) - g(\mathbf{u}) - H(\mathbf{u}) (\mathbf{x} - \mathbf{u}) \|_{\mathbf{u}}^* \leq \frac{r^2}{1-r}. \]
    \end{enumerate}
\end{proposition}

\begin{proposition}[\protect{\cite[Chapters 2--3]{Renegar2001}, \cite[Chapter 5]{Nesterov2018}}]
    Let $f$ be a $\nu$-LHSCB for the proper convex cone $\mathcal{K}$ and let $\vx\in\K^{\circ}$. Then:
       \setlist[enumerate,1]{label=\textup{(A\arabic*)}, leftmargin=5em, labelsep=0.5em, itemindent=0pt, start=4}
    \begin{enumerate}
        \item \label{prop:lhscb_Hessian} $H(\mathbf{x}) \mathbf{x} = - g(\mathbf{x})$;
        \item \label{prop:lhscb_gradient} $ g(\mathbf{x})^{\top} \mathbf{x} = - \nu$;
        \item \label{prop:lhscb_lh} $g( \alpha \mathbf{x}) = \alpha^{-1} g(\mathbf{x})$ and $H( \alpha \mathbf{x} ) = \alpha^{-2} H(\mathbf{x})$ for all $\alpha > 0$;
        \item  \label{prop:lhscb_norm} $\| g(\mathbf{x}) \|_x^* = \sqrt{\nu}$.
    \end{enumerate}
\end{proposition}

\section{Stopping Algorithm \ref{alg:simple_short_step} with a damped Newton step}
\label{sec:alg3_damped}

Here we collect a few technical results about a variant of Algorithm \ref{alg:simple_short_step} with damped Newton steps, used in the analysis of the Phase 1 method of Section \ref{sec:two-phase}.

\begin{lemma}
\label{lemma:damped_2eta}
    Let $(\mathbf{x}, \mathbf{y}, \mathbf{s}) \in \mathcal{N}(\eta, \tau)$ and let $(\Delta \mathbf{x}, \Delta \mathbf{y}, \Delta \mathbf{s})$ be the solution of \eqref{eq:NS}. Let $\mathbf{x}^+ := \mathbf{x} + \alpha \Delta \mathbf{x}$, $\mathbf{s}^+ := \mathbf{s} + \alpha \Delta \mathbf{s}$, and $\mathbf{y}^+ := \mathbf{y} + \alpha \Delta \mathbf{y}$, where $\alpha \in (0, 1]$. Let $\tau^+ := (1-\vartheta)\tau$, where $\vartheta= \frac{\eta/2}{\sqrt{\nu} + 1}$ and $0<\eta \leq 1/4$ as before in Algorithm \ref{alg:simple_short_step}. Then $\| \mathbf{s}^+ + \tau^+ g(\mathbf{x}^+) \|_{\mathbf{x}^+}^* \leq 2 \eta \tau^+$.
\end{lemma}

\begin{proof}
    Using \ref{prop:scb_norm1} and \ref{sd_prop:second} of Lemma \ref{lemma:sd_properties}, we have
\begin{equation} \label{eq:damped_xpnorm}
\| \mathbf{s}^+ + \tau g(\mathbf{x}^+) \|_{\mathbf{x}^+}^* \leq \frac{\| \mathbf{s}^+ + \tau g(\mathbf{x}^+) \|_{\mathbf{x}}^*}{1 - \alpha \eta},
\end{equation}
as $\| \mathbf{x}^+ - \mathbf{x} \|_{\mathbf{x}} = \alpha \|\Delta \mathbf{x}\|_{\mathbf{x}} \leq \alpha \eta < 1$ by part \ref{sd_prop:second} of Lemma \ref{lemma:sd_properties} and $\eta \leq 1/4$.

Applying \ref{prop:scb_norm2} and \ref{sd_prop:second} of Lemma \ref{lemma:sd_properties} again gives
\begin{equation} \label{eq:g_exp_bound}
\| g(\mathbf{x}^+) - g(\mathbf{x}) - \alpha H(\mathbf{x}) \Delta \mathbf{x} \|_{\mathbf{x}}^* \leq \frac{\alpha^2 \eta^2}{1 - \alpha \eta}.
\end{equation}

Using (\ref{eq:g_exp_bound}), the third equation of \eqref{eq:NS} and the neighborhood definition, we can derive the following upper bound:
\[
\begin{aligned}
\| \mathbf{s}^+ + \tau g(\mathbf{x}^+) \|_{\mathbf{x}}^* &\leq \| \mathbf{s} + \alpha \Delta \mathbf{s} + \tau (g(\mathbf{x}) + \alpha H(\mathbf{x}) \Delta \mathbf{x}) \|_{\mathbf{x}}^* \\ &\qquad\qquad + \tau \| g(\mathbf{x}^+) - g(\mathbf{x}) - \alpha H(\mathbf{x}) \Delta \mathbf{x} \|_{\mathbf{x}}^* \\
&= \| \mathbf{s} + \tau g(\mathbf{x}) + \alpha (\Delta \mathbf{s} + \tau H(\mathbf{x}) \Delta \mathbf{x}) \|_{\mathbf{x}}^* \\ &\qquad\qquad + \tau \| g(\mathbf{x}^+) - g(\mathbf{x}) - \alpha H(\mathbf{x}) \Delta \mathbf{x} \|_{\mathbf{x}}^* \\
&\stackrel{\eqref{eq:NS}}{=} (1 - \alpha) \| \mathbf{s} + \tau g(\mathbf{x}) \|_{\mathbf{x}}^* + \tau \| g(\mathbf{x}^+) - g(\mathbf{x}) - \alpha H(\mathbf{x}) \Delta \mathbf{x} \|_{\mathbf{x}}^* \\
&\stackrel{\eqref{eq:g_exp_bound}}{\leq} (1 - \alpha) \eta \tau + \tau \frac{\alpha^2 \eta^2}{1 - \alpha \eta} = \tau \left((1 - \alpha)\eta + \frac{\alpha^2 \eta^2}{1 - \alpha \eta} \right).
\end{aligned}
\]
From \eqref{eq:damped_xpnorm} and the previously derived upper bound it follows that
\begin{equation} \label{eq:2eta_xplusnorm}
\| \mathbf{s}^+ + \tau g(\mathbf{x}^+) \|_{\mathbf{x}^+}^* \stackrel{\eqref{eq:damped_xpnorm}}{\leq} \frac{\| \mathbf{s}^+ + \tau g(\mathbf{x}^+) \|_{\mathbf{x}}^*}{1 - \alpha \eta} \leq \tau \left[ \frac{(1 - \alpha) \eta}{1 - \alpha \eta} + \frac{\alpha^2 \eta^2}{(1 - \alpha \eta)^2} \right].
\end{equation}
Therefore,

\begin{gather}
\frac{1}{\tau^+} \| \mathbf{s}^+ + \tau^+ g(\mathbf{x}^+) \|_{\mathbf{x}^+}^* \leq \frac{1}{(1 - \vartheta) \tau} \left( \| \mathbf{s}^+ + \tau g(\mathbf{x}^+) \|_{\mathbf{x}^+}^* + \vartheta \tau \| g(\mathbf{x}^+) \|_{\mathbf{x}^+}^* \right) \notag \\
\stackrel{\eqref{eq:2eta_xplusnorm},\ref{prop:lhscb_norm}}{\leq} \frac{1}{1 - \vartheta} \left( \frac{(1 - \alpha) \eta}{1 - \alpha \eta} + \frac{\alpha^2 \eta^2}{(1 - \alpha \eta)^2} + \vartheta \sqrt{\nu} \right). \label{eq:damped_xpnorm2}
\end{gather}
In the analysis of Algorithm \ref{alg:simple_short_step}, we have $\eta \leq \frac{1}{4}$, and $
\vartheta = \frac{\eta/2}{\sqrt{\nu} + 1}$, thus,
\[
\frac{1}{1 - \vartheta} =  \frac{\sqrt{\nu} + 1}{\sqrt{\nu} + 1 - \eta/2} \leq \frac{2}{2 - \eta/2} \text{ and }
\frac{\vartheta \sqrt{\nu}}{1 - \vartheta} = \frac{\frac{\eta}{2} \sqrt{\nu}}{\sqrt{\nu} + 1 - \eta/2} \leq \frac{\eta}{2}.
\]
Using these inequalities and \eqref{eq:damped_xpnorm2}, we can derive the upper bound in the statement of the lemma:

\[
\begin{aligned}
\frac{1}{\tau^+} \| \mathbf{s}^+ + \tau^+ g(\mathbf{x}^+) \|_{\mathbf{x}^+}^*
 &\leq \frac{2}{2 - \eta/2} \left( \frac{(1 - \alpha) \eta}{1 - \alpha \eta} + \frac{\alpha^2 \eta^2}{(1 - \alpha \eta)^2} \right) + \frac{\eta}{2} \\
 &\leq \frac{2 \eta}{2-\eta/2} + \frac{\eta}{2} \leq 2 \eta,
\end{aligned}
\]
as $\frac{(1 - \alpha) \eta}{1 - \alpha \eta} + \frac{\alpha^2 \eta^2}{(1 - \alpha \eta)^2}$ is strictly decreasing in $\alpha$ if $\eta \leq \frac{1}{4}$ holds.
\end{proof}

\begin{lemma} \label{lemma:alg3_mu_bounds}
Suppose Algorithm \ref{alg:simple_short_step} is stopped after a damped Newton step of size $\alpha \in (0,1]$, following a series of full Newton steps. Then the final iterate satisfies
\[
\left( 1 - \frac{\eta^2}{\nu} \right) \tau \leq \mu \leq \frac{1}{1 - \vartheta} \tau.
\]
\end{lemma}

\begin{proof}
   Let us denote the penultimate iterate by $(\mathbf{x}, \vy, \mathbf{s})$ and the last iterate by
   \[(\mathbf{x}^+, \vy^+, \mathbf{s}^+) = (\mathbf{x}, \vy, \mathbf{s}) + \alpha (\Delta\vx,\Delta\vy,\Delta\vs).\]
The final duality gap is
$
\mu = \frac{(\mathbf{x}^+)^\top \mathbf{s}^+}{\nu}.
$
Proceeding analogously to the proof of Lemma \ref{lemma:DG}, we can write
\[
(\mathbf{x}^+)^\top \mathbf{s}^+ = \mathbf{x}^\top \mathbf{s} - \alpha \tau \|\Delta \mathbf{x}\|_{\mathbf{x}}^2 - \alpha \mathbf{x}^\top \mathbf{s}
+ \alpha \tau \nu = (1 - \alpha) \mathbf{x}^\top \mathbf{s} + \alpha \tau \left( \nu - \|\Delta \mathbf{x}\|_{\mathbf{x}}^2 \right).
\]
From Lemma \ref{lemma:DG}, after the penultimate step we have
\begin{equation} \label{eq:dgbound}
\frac{\tau}{1-\vartheta} (\nu-\eta^2) \leq \mathbf{x}^\top \mathbf{s} \leq \frac{\tau}{1-\vartheta} \nu.
\end{equation}
Using (\ref{eq:dgbound}) and \ref{sd_prop:second} from Lemma \ref{lemma:sd_properties}, we can derive the desired lower and an upper bounds on $\mu$:
\[
\mu = \frac{(\mathbf{x}^+)^\top \mathbf{s}^+}{\nu} \leq \frac{(1 - \alpha) \tau }{1 - \vartheta} + \alpha \tau \leq \frac{\tau}{1 - \vartheta},
\]
and
\[
\mu = \frac{(\mathbf{x}^+)^\top \mathbf{s}^+}{\nu} \geq \frac{(1 - \alpha) \tau \left( 1 - \frac{\eta^2}{\nu} \right)}{1 - \vartheta} + \alpha \tau \left( 1 - \frac{\eta^2}{\nu} \right) \geq \tau \left( 1 - \frac{\eta^2}{\nu} \right). \qedhere
\]
\end{proof}

\end{document}